\documentclass{aims}

\newcommand{\vA}{{\mathbf{A}}}
\newcommand{\vB}{{\mathbf{B}}}
\newcommand{\vD}{{\mathbf{D}}}
\newcommand{\vI}{{\mathbf{I}}}
\newcommand{\vL}{{\mathbf{L}}}
\newcommand{\vM}{{\mathbf{M}}}
\newcommand{\vS}{{\mathbf{S}}}

\newcommand{\vU}{{\mathbf{U}}}
\newcommand{\vV}{{\mathbf{V}}}
\newcommand{\vX}{{\mathbf{X}}}
\newcommand{\vY}{{\mathbf{Y}}}
\newcommand{\vZ}{{\mathbf{Z}}}
\newcommand{\cA}{{\mathcal{A}}}

\newcommand{\RR}{\mathbb{R}}
\newcommand{\sign}{\mathrm{sign}}
\newcommand{\vzero}{\mathbf{0}}
\newcommand{\St}{{\mathrm{subject~to}}} 
\newcommand{\diag}{{\mathrm{diag}}} 
\newcommand{\rank}{{\mathrm{rank}}} 
\newcommand{\prox}{\mathbf{prox}}
\DeclareMathOperator*{\argmin}{arg\,min}
\DeclareMathOperator*{\Min}{minimize}

\usepackage{caption}
\usepackage{subcaption}
\usepackage{multirow}
\usepackage[ruled,vlined]{algorithm2e}

\usepackage{amsmath}
\usepackage{paralist}
\usepackage{graphics} 
\usepackage{epsfig} 
\usepackage{graphicx}  \usepackage{epstopdf}
\usepackage[colorlinks=true]{hyperref}
\hypersetup{urlcolor=blue, citecolor=red}

\def\currentvolume{X}
 \def\currentissue{X}
  \def\currentyear{200X}
   \def\currentmonth{XX}
    \def\ppages{X--XX}
     \def\DOI{10.3934/xx.xxxxxxx}

\newtheorem{theorem}{Theorem}[section]

\newtheorem{lemma}[theorem]{Lemma}

\theoremstyle{definition}

\title[Fast algorithms for robust principal component analysis] 
{Fast algorithms for robust principal component analysis with an upper bound on the rank}

\author[Ningyu Sha, Lei Shi, and Ming Yan]{}

\subjclass{Primary: 65K10, 90C26; Secondary: 65D18.}
\keywords{Robust principal component analysis, nonconvex, acceleration, low-rank, sparse.}

\email{shaningy@msu.edu}
\email{leishi@fudan.edu.cn}
\email{myan@msu.edu}

\thanks{N. Sha and M. Yan are supported by NSF grant DMS-1621798 and DMS-2012439. L. Shi is supported by NNSFC grant 11631015 and Shanghai Science and Technology Research Program 19JC1420101.}

\thanks{$^*$ Corresponding author.}

\begin{document}
\maketitle

\centerline{\scshape Ningyu Sha}
\medskip
{\footnotesize
	\centerline{Department of Computational Mathematics, Science and Engineering}
	\centerline{Michigan State University, East Lansing, MI 48824, USA}
} 
\medskip

\centerline{\scshape Lei Shi}
\medskip
{\footnotesize
	\centerline{School of Mathematical Sciences, Shanghai Key Laboratory for Contemporary Applied Mathematics}
	\centerline{Key Laboratory of Mathematics for Nonlinear Sciences (Fudan University), Ministry of Education}
	\centerline{Fudan University, Shanghai, 200433, P.R. China}
} 

\medskip

\centerline{\scshape Ming Yan$^*$}
\medskip
{\footnotesize
	\centerline{Department of Computational Mathematics, Science and Engineering, Department of Mathematics}
	\centerline{Michigan State University, East Lansing, MI 48824, USA}
}

\bigskip

\centerline{(Communicated by the associate editor name)}

\begin{abstract}
The robust principal component analysis (RPCA) decomposes a data matrix into a low-rank part and a sparse part. There are mainly two types of algorithms for RPCA. The first type of algorithm applies regularization terms on the singular values of a matrix to obtain a low-rank matrix. However, calculating singular values can be very expensive for large matrices. The second type of algorithm replaces the low-rank matrix as the multiplication of two small matrices. They are faster than the first type because no singular value decomposition (SVD) is required. However, the rank of the low-rank matrix is required, and an accurate rank estimation is needed to obtain a reasonable solution. In this paper, we propose algorithms that combine both types. Our proposed algorithms require an upper bound of the rank and SVD on small matrices. First, they are faster than the first type because the cost of SVD on small matrices is negligible. Second, they are more robust than the second type because an upper bound of the rank instead of the exact rank is required. Furthermore, we apply the Gauss-Newton method to increase the speed of our algorithms. Numerical experiments show the better performance of our proposed algorithms.   
\end{abstract}

\section{Introduction}
Robust principal component analysis (RPCA) decomposes a data matrix into a low-rank part and a sparse part. It has applications in a wide range of areas, including computer vision~\cite{de2001robust}, image processing~\cite{liu2012robust,elhamifar2013sparse}, dimensionality reduction~\cite{cunningham2015linear}, and bioinformatics data analysis~\cite{da2009weighted}. More specifically, the RPCA model has achieved great success in video surveillance and face recognition~\cite{candes2011robust,bouwmans2014robust}. For example, in video surveillance, the low-rank part preserves the stationary background, whereas the sparse part can capture a moving object or person in the foreground. 

We first assume that the data matrix $\vD$ is obtained by the sum of a low-rank matrix and a sparse matrix. That is 
\begin{align*}
\vD = \vL + \vS,
\end{align*}
where $\vL$ is a low-rank matrix and $\vS$ is a sparse matrix having only a few nonzero entries. RPCA is an inverse problem to recover $\vL$ and $\vS$ from the matrix $\vD$, which can be realized via solving the idealized nonconvex problem
\begin{equation}\label{pro:L0L0C}
\Min_{\vL, \vS}~\rank(\vL) + \lambda \|\vS\|_0, ~\St~\vL + \vS = \vD,
\end{equation}
where $\lambda$ is a parameter to balance the two objectives and $\|\vS\|_0$ counts the number of non-zero entries in $\vS$. However, this problem is NP-hard in general~\cite{amaldi1998approximability}. Therefore, much attention is focused on the following convex relaxation:
\begin{equation}\label{pro:L1L1C}
\Min_{\vL, \vS}~\|\vL\|_* + \lambda \|\vS\|_1, ~\St~\vL + \vS = \vD.
\end{equation} 
Here $\|\cdot\|_*$ and $\|\cdot\|_1$ denote the nuclear norm and $\ell_1-$norm of a matrix, respectively. It is shown that under mild conditions, the convex model~\eqref{pro:L1L1C} can exactly recover the low-rank and sparse parts with high probabilities~\cite{candes2011robust}.  When additional Gaussian noise is considered, we can set the noise level to be $\epsilon$ and use the Frobenius norm $\|\cdot\|_F$ to measure the reconstruction error. Then, the problem becomes 
\begin{equation}\label{pro:L1L1C2}
\Min_{\vL, \vS}~\|\vL\|_* + \lambda \|\vS\|_1, ~\St~\|\vL + \vS - \vD\|_F^2\leq \epsilon.
\end{equation}
This constrained optimization problem is equivalent to the unconstrained problem
\begin{equation}\label{pro:L1L1U1}
\Min_{\vL, \vS }~\frac{\mu}{2} \|\vL+\vS-\vD\|_F^2 + \mu\|\vL\|_* + \lambda\mu \|\vS\|_1
\end{equation} 
with a trade-off parameter $\mu$. There is a correspondence between the two parameters $\epsilon$ and $\mu$ in~\eqref{pro:L1L1C2} and~\eqref{pro:L1L1U1}, but the explicit expression does not exist. In this paper, we will focus on the unconstrained problem~\eqref{pro:L1L1U1}, and the technique introduced in this paper can be applied to the convex models~\eqref{pro:L1L1C} and~\eqref{pro:L1L1C2}. Please see Section~\ref{section4} for more details.

There are many existing approaches for solving~\eqref{pro:L1L1U1}, including the augmented Lagrange method~\cite{lin2010augmented,bouwmans2014robust,wright2009robust}. Some examples are proximal gradient method for $(\vL,\vS)$, alternating minimization for $\vL$ and $\vS$~\cite{Shen018}, proximal gradient method for $\vL$ after $\vS$ is eliminated~\cite{sha2019efficient}, alternating direction method of multipliers (ADMM)~\cite{yuan2009sparse,tao2011recovering}. All these approaches need to find the proximal of the nuclear norm, which requires singular value decomposition (SVD). When the matrix size is large, the SVD computation is very expensive and dominates other computation~\cite{trefethen1997numerical}. 

Alternative approaches for RPCA use matrix decomposition~\cite{wen2012solving} and do not require SVD. Assuming that the rank of $\vL$ is known as $p$, we can decompose it as 
$$\vL=\vX\vY^\top,$$
with $\vX\in\RR^{m\times p}$ and $\vY\in\RR^{n\times p}$. Then the following nonconvex optimization problem 
\begin{equation}\label{pro:L1L1U}
\Min_{\vX,\vY,\vS }~\frac{1}{2} \|\vX\vY^\top+\vS-\vD\|_F^2 + \lambda \|\vS\|_1,
\end{equation}
is considered. There are infinite many optimal solutions for this problem, since for any invertable matrix $\vA\in\RR^{p\times p}$, $(\vX,\vY,\vS)$ and $(\vX\vA^{-1},\vY\vA^{\top},\vS)$ have the same objective value. In fact, for any matrix $\vL$ with rank no greater than $p$, we can find $\vL=\vX\vY^\top$ with $\vY^\top\vY=\vI_{p\times p}$. Therefore, we can have an additional constraint $\vY^\top\vY=\vI_{p\times p}$. The resulting problem still has infinite many optimal solutions, since for any orthogonal matrix $\vA\in\RR^{p\times p}$, $(\vX,\vY,\vS)$ and $(\vX\vA,\vY\vA,\vS)$ have the same objective value. Though $(\vX,\vY)$ are not unique, the low-rank matrix $\vL=\vX\vY^\top$ that we need could be unique. This resulting problem was discussed in~\cite{Shen018}, and an efficient algorithm by alternating minimizing $\vX\vY^\top$ and $\vS$ is provided. In this algorithm, a Gauss-Newton algorithm is applied to update $\vX\vY^\top$ and reduce the time.

Though the matrix decomposition approach could be solved faster than the nuclear norm minimization approach because no SVD is required, it is nonconvex and requires an accurate estimation of the rank of $\vL$. Fig.~\ref{fig:robust} in Section~\ref{sec:robust} demonstrates that a good estimation of the rank is critical. However, in most scenarios, we do not have the exact rank of $\vL$, but we can have an upper bound of the true rank. Therefore, we can combine the matrix decomposition and the nuclear norm minimization to have the benefits of both approaches: fast speed and robustness in the rank. The problem we consider in this paper is 
\begin{align}\label{pro:model2}
\Min_{\vL,\vS }~\frac{1}{2} \|\vL+\vS-\vD\|_F^2 + \mu\|\vL\|_*+ \lambda \|\vS\|_1, ~\St~ \mbox{rank}(\vL)\leq p.
\end{align}
When $\mu=0$, the problem~\eqref{pro:model2} is equivalent to~\eqref{pro:L1L1U}. In addition, we consider the following more general problem
\begin{align}\label{pro:general}
\Min_{\vL,\vS }~\frac{1}{2} \|\mathcal{A}(\vL)+\vS-\vD\|_F^2 + \mu\|\vL\|_*+ \lambda \|\vS\|_1, ~\St~ \mbox{rank}(\vL)\leq p,
\end{align}
where $\vD$ is the measurement of $\mathcal{A}(\vL)$ contaminated with both Gaussian noise and sparse noise. Here $\mathcal{A}$ is a bounded linear operator that describes how the measurements are calculated. For example, in robust matrix completion, we let $\mathcal{A}$ be the restriction operator on the given components of the matrix $\vL$. 

Note that the alternating minimization algorithm in~\cite{Shen018} can not be applied to this general problem because the subproblem for $\vL$ can no longer be solved efficiently by the Gauss-Newton method. We will show the equivalency of the alternating minimization algorithm in~\cite{Shen018} and a proximal gradient method applied to a problem with $\vL$ only. Then the subproblem of $\vL$ in our general problem~\eqref{pro:general} can still be solved efficiently with the Gauss-Newton method. Please see more details in Section~\ref{section2}.

For simplicity, we use the nuclear norm and $\ell_1-$norm for the low-rank and sparse matrices, respectively. The main purpose of this paper is to introduce a fast algorithm to solve~\eqref{pro:general}. Though the technique can be applied to variants of~\eqref{pro:general}, as will be shown in Section~\ref{section4}, the comparison of different penalties is out of the scope of this paper. The contributions of this paper are: 
\begin{itemize}
	\item We propose a new model for RPCA, which combines the nuclear norm minimization and the matrix decomposition. The matrix decomposition brings efficient algorithms, and the nuclear norm minimization on a smaller matrix removes the requirement of the rank of the low-rank matrix. Note that other nonconvex penalties can replace the nuclear norm minimization, and the results in this paper are still valid. 
	\item We develop efficient algorithms using Gauss-Newton to solve this problem and show its convergence.
\end{itemize}

\subsection{Notation}
Throughout this paper, matrices are denoted by bold capital letters (e.g., $\vA$), and operators are denoted by calligraphic letters (e.g., $\mathcal{A}$). In particular, $\vI$ denotes the identity matrix, $\vzero$ denotes the zero matrix (all entries equal zero), and $\mathcal{I}$ denotes the identity operator. If there is potential for confusion, we indicate the dimension of matrix with subscripts. For a matrix $\vA$, $\vA^\top$ represents its transpose and $\vA(:,j:k)$ denotes the matrix composed by the columns of $\vA$ indexing from $j$ to $k$. Let $\vA_{i,j}$ be the $(i,j)$ entry of $\vA$. The $\ell_1-$norm of $\vA$ is given by $\|\vA\|_1=\sum_{i,j}|\vA_{i,j}|$.  We denote the $i$th singular value of $\vA$  by $\sigma_i(\vA)$. The nuclear norm of $\vA$ is given by $\|\vA\|_*=\sum_i \sigma_i(\vA)$. We will use $\partial \|\cdot\|_1$  and $\partial \|\cdot\|_*$ to denote the subgradients of $\ell_1-$norm and nuclear norm, respectively. The linear space of all $m\times n$ real matrices is denoted by $\RR^{m\times n}$. For $\vA, \vB\in \RR^{m\times n}$, the inner product of $\vA,\vB$ is defined by $\langle \vA, \vB \rangle=\mbox{Tr}(\vA^\top \vB)$, which induces the Frobenius norm $\|\vA\|_F=\sqrt{\mbox{Tr}(\vA^\top \vA)}=\sqrt{\sum_i \sigma^2_i(\vA)}$. Let $\mathcal{A}$ be a linear bounded operator on $\RR^{m\times n}$. The operator norm of $\mathcal{A}$ is given by $\|\mathcal{A}\|=\sup\{\|\mathcal{A}(\vA)\|_F:\vA \in \mathbb{R}^{m \times n},\|\vA\|_F=1\}$. The adjoint operator of $\mathcal{A}$ denoted by $\mathcal{A}^*$ is also linear and bounded on $\RR^{m\times n}$ such that $\langle \mathcal{A}(\vA), \vB \rangle=\langle \vA, \mathcal{A}^*(\vB) \rangle$. Notation $\odot$ is used to denote the component-wise multiplication. Additionally, for a function $f:\RR \to \RR$, without further reference, $f$ acting on a matrix $\vA\in \RR^{m\times n}$ specifies that $f$ is evaluated on each entry of $\vA$, i.e., $f(\vA)\in \RR^{m \times n}$ with $(f(\vA))_{i,j}=f(\vA_{i,j})$. For example, if $f(x)=|x|-\lambda$, we can denote $f(\vA)\in \RR^{m\times n}$ by $|\vA|-\lambda$ with $(|\vA|-\lambda)_{i,j}=|\vA_{i,j}|-\lambda$.

\subsection{Organization} 
The rest of the paper is organized as follows. We introduce our proposed algorithms and show their convergence in Section~\ref{section2}. Then we conduct numerical experiments to compare our proposed algorithms' performance with existing approaches in Section~\ref{sec:num}. In Section~\ref{section4}, we conclude this paper with some potential extensions.

\section{Proposed algorithms}\label{section2}
The problem~\eqref{pro:model2} is nonconvex because of the constraint $\mbox{rank}(\vL)\leq p$. It has several equivalent formulations. E.g., it is equivalent to the following nonconvex weighted nuclear norm minimization problem:
\begin{align*}
\Min_{\vL,\vS }~\frac{1}{2} \|\vL+\vS-\vD\|_F^2 + \mu\sum_{i=1}^p\sigma_i(\vL)+C\sum_{i=p+1}^{\min(m,n)}\sigma_i(\vL)+ \lambda \|\vS\|_1,
\end{align*}
where $C$ is a sufficiently large number such that the optimal $\vL$ has at most $p$ nonzero singular values. However, this formulation also requires the singular value decomposition of a $m\times n$ matrix in each iteration, which is expensive when $m$ and $n$ are large. We consider another equivalent problem with matrix decomposition in the following theorem.
\begin{theorem}\label{Thm2.1}
	Problem~\eqref{pro:model2} is equivalent to 
	\begin{equation}\label{pro:model}
	\Min_{\vX,\vY,\vS }~\frac{1}{2} \|\vX\vY^\top+\vS-\vD\|_F^2 + \mu\|\vX\|_*+ \lambda \|\vS\|_1, ~\St~ \vY^\top\vY=\vI_{p\times p}.
	\end{equation}
	More specifically, if $(\vX,\vY,\vS)$ is an optimal solution to~\eqref{pro:model}, then $(\vX\vY^\top,\vS)$ is an optimal solution to~\eqref{pro:model2}. If $(\vL,\vS)$ is an optimal solution to~\eqref{pro:model2} and we have the decomposition $\vL=\vX\vY^\top$ with $\vY^\top\vY=\vI_{p\times p}$, then $(\vX,\vY,\vS)$ is an optimal solution to~\eqref{pro:model}.
\end{theorem}
\begin{proof}
	For any matrix $\vL\in\RR^{m\times n}$ with rank no greater than $p$, we can have the decomposition 
	$$\vL=\vX\vY^\top,$$
	with $\vY^\top\vY=\vI_{p\times p}.$ 
	This decomposition is not unique, and one decomposition can be easily obtained from the compact SVD of $\vL$. 
	Let $\vL=\vU_p\Sigma_p\vV^\top_p$ be the SVD of $\vL$ with a square $p\times p$ matrix $\Sigma_p$, we have $\vV_p^\top\vV_p=\vI_{p\times p}$.
	Thus, problem~\eqref{pro:model2} is equivalent to 
	\begin{align*}
	\Min_{\vX,\vY,\vS }~\frac{1}{2} \|\vX\vY^\top+\vS-\vD\|_F^2 + \mu\|\vX\vY^\top\|_*+ \lambda \|\vS\|_1, ~\St~ \vY^\top\vY=\vI_{p\times p}.
	\end{align*}
	For any $\vX\in\RR^{m\times p}$, let $\vX=\vU\Sigma\vV^\top$ be its SVD with $\vU\in\RR^{m\times p}$ and $\vV\in\RR^{p\times p}$. We have 
	$$\vX\vY^\top = \vU\Sigma\vV^\top\vY^\top=\vU\Sigma(\vY\vV)^\top.$$
	Since $(\vY\vV)^\top(\vY\vV)=\vV^\top\vY^\top\vY\vV=\vV^\top\vV=\vI_{p\times p}$. The SVD of $\vX\vY^\top$ is $\vU\Sigma(\vY\vV)^\top$, and $\|\vX\vY^\top\|_* =\sum_{i=1}^p\Sigma_{ii}=\|\vX\|_*$. 
	Thus, problem~\eqref{pro:model2} is equivalent to~\eqref{pro:model}.
\end{proof}

Next, we consider problem~\eqref{pro:model} with $\vS$ fixed. When $\vS$ is fixed, it becomes a problem of $\vL=\vX\vY^\top$, and solving this problem is to find the proximal operator of the corresponding nonconvex weighted nuclear norm, which is denoted as
\begin{equation}\label{pro:model2_fixedS}
\Min_{\vL }~\frac{1}{2} \|\vL-\vM\|_F^2 + \mu\|\vL\|_*, ~\St~ \mbox{rank}(\vL)\leq p,
\end{equation} or equivalently
\begin{equation}\label{pro:model_fixedS}
\Min_{\vX,\vY }~\frac{1}{2} \|\vX\vY^\top-\vM\|_F^2 + \mu\|\vX\|_*, ~\St~ \vY^\top\vY=\vI_{p\times p},
\end{equation} where $\vM=\vD-\vS$.

\begin{theorem}\label{thm:proximal} Let $q=\min(m,n)$. Problem~\eqref{pro:model2_fixedS} can be solved in two steps:
	\begin{enumerate}
		\item Find the compact SVD of $\vM=\vU\Sigma\vV^\top$, with 	    $\Sigma=\diag(\sigma_1(\vM),\cdots,\sigma_{q}(\vM))$ satisfying $\sigma_1(\vM)\geq \sigma_2(\vM)\geq \cdots\geq \sigma_q(\vM)$;
		\item Construct a diagonal matrix $\hat\Sigma_\mu\in\RR^{p\times p}$ with $(\hat\Sigma_\mu)_{ii}=\max(\Sigma_{ii}-\mu,0)$, then one solution of~\eqref{pro:model2_fixedS} is $\vU(:,1:p)\hat\Sigma_{\mu}\vV(:,1:p)^\top$.
	\end{enumerate}
	In addition, for any orthogonal matrix $\vA\in\RR^{p\times p}$, $(\vU(:,1:p)\hat\Sigma_{\mu}\vA, \vV(:,1:p)\vA)$ is an optimal solution of~\eqref{pro:model_fixedS}.
\end{theorem}
\begin{proof}
	Given any $\vL\in\RR^{m\times n}$ with $\mbox{rank}(\vL)\leq p$, let $\sigma_1,\sigma_2,\cdots,\sigma_q$ be its singular values in the decreasing order such that $\sigma_{p+1}=\cdots=\sigma_{q}=0$. Note that the main diagonal entries of $\Sigma$ are the singular values of $\vM$. According to von-Neumann trace inequality~\cite[Theorem 7.4.1.1]{horn2012matrix}, one can bound the matrix inner product by the singular values, i.e., $\langle\vL, \vM\rangle\leq \sum_{i=1}^q \sigma_i\Sigma_{ii}$. Then we have 
	\begin{eqnarray}\label{eq1}
	\begin{aligned}
	\frac{1}{2} \|\vL-\vM\|_F^2 + \mu\|\vL\|_*
	= & \frac{1}{2}\|\vL\|_F^2+\frac{1}{2}\|\vM\|_F^2-\langle\vL, \vM\rangle+\mu\|\vL\|_*\\
	\geq & \frac{1}{2}\sum_{i=1}^q \sigma^2_i+\frac{1}{2}\sum_{i=1}^q\Sigma^2_{ii}-\sum_{i=1}^q \sigma_i\Sigma_{ii}+\mu \sum_{i=1}^q\sigma_i\\
	= & \frac{1}{2}\sum_{i=1}^p \sigma^2_i+\frac{1}{2}\sum_{i=1}^q\Sigma^2_{ii}-\sum_{i=1}^p \sigma_i\Sigma_{ii}+\mu \sum_{i=1}^p\sigma_i,    
	\end{aligned}
	\end{eqnarray}
	where the equality is satisfied when $\vL$ has a simultaneous SVD with $\vM$ through $\vU$ and $\vV$. Therefore, the optimal $\vL$ minimizing $\frac{1}{2} \|\vL-\vM\|_F^2 + \mu\|\vL\|_*$ can be selected from the matrices that have a simultaneous SVD with $\vM$ through $\vU$ and $\vV$. Then we can assume that the optimal $\vL$ satisfies $$\vL=\vU\mbox{diag}(\sigma_1,\cdots,\sigma_p,\sigma_{p+1},\cdots,\sigma_q)\vV^\top=\vU(:,1:p)\mbox{diag}(\sigma_1,\cdots,\sigma_p)\vV(:,1:p)^\top,$$ where the last equality holds because of the fact that $\sigma_{p+1}=\cdots=\sigma_{q}=0$. Next, one can construct an optimal $\vL$ of the above form by letting $\sigma_i=\max(\Sigma_{ii}-\mu,0)$ for $i=1,2,\cdots,p$, which minimizes the last equation in ~\eqref{eq1}. Thus $\vU(:,1:p)\hat\Sigma_{\mu}\vV(:,1:p)^\top$ minimizes the objective function of ~\eqref{pro:model2_fixedS} over all $\vL\in\RR^{m\times n}$ with rank no greater than $p$.
	
	By the same argument in the proof of Theorem \ref{Thm2.1}, we see that problem~\eqref{pro:model_fixedS} is equivalent to problem~\eqref{pro:model2_fixedS}. Since for any orthogonal matrix $\vA\in \RR^{p\times p}$, there hold
	$$\vL=(\vU(:,1:p)\hat\Sigma_{\mu}\vA)(\vV(:,1:p)\vA)^\top$$ and 
	$$(\vV(:,1:p)\vA)^\top(\vV(:,1:p)\vA)=\vA^\top \vA=\vI_{p\times p}.$$ Therefore, $(\vU(:,1:p)\hat\Sigma_{\mu}\vA,\vV(:,1:p)\vA)$ is an optimal solution of problem~\eqref{pro:model_fixedS}.\end{proof}

The first step to solve problem~\eqref{pro:model_fixedS} in the previous theorem requires the truncated SVD of an $m\times n$ matrix $\vM$. Since we only need the first $p$ ($p< q=\min(m,n)$) singular values, we use the Gauss-Newton algorithm to find $(\vX,\vY)$ alternatively. In this approach, we require the SVD of a $m\times p$ matrix, which is much faster than the truncated SVD of a $m\times n$ matrix when $p$ is small. In addition, we use the previous $\vX$ as the initial guess in the next iteration to reduce the number of inner iterations for the Gauss-Newton algorithm.

\begin{lemma}\label{lem2.3}
	If the rank of $\vM\in\RR^{m\times n}$ is larger than $p$, problem~\eqref{pro:model_fixedS} can be solved in the following three steps:
	\begin{enumerate}
		\item Find $\hat\vX\in\RR^{m\times p}$ ($p<m$) by solving the following optimization problem $$\Min_{\vX}~\frac{1}{2} \|\vX\vX^\top-\vM\vM^\top\|_F^2;$$
		\item $\vY=\vM^\top\hat\vX(\hat\vX^\top\hat\vX)^{-1}$;
		\item Let $\hat\vX=\vU_p\hat\Sigma\vA$ be its thin SVD with $\hat\Sigma\in\RR^{p\times p}$ and choose $\vX$ as $\vX=\vU_p\hat\Sigma_\lambda\vA$ with $(\hat\Sigma_\lambda)_{ii}=\max(0,\hat\Sigma_{ii}-\mu)$ for $i=1,\dots,p$. Then $(\vX,\vY)$ is an solution of problem~\eqref{pro:model_fixedS}.  
	\end{enumerate}
	
\end{lemma}
\begin{proof}
	Given any $\vX \in \RR^{m\times p}$, let $\lambda_1,\lambda_2,\cdots,\lambda_m$ be the non-negative eigenvalues of the matrix $\vX \vX^\top$. Since $\mbox{rank}(\vX)\leq p <m$, we have $\lambda_{p+1}=\cdots=\lambda_m=0$. Recall that the compact SVD of $\vM$ given in Theorem~\ref{thm:proximal} is $\vU\Sigma\vV^\top$ with $\Sigma\in\RR^{q\times q}$ (here $q=\min(m,n)$). Then $\Sigma^2_{11}\geq \Sigma^2_{11}\geq \cdots\geq \Sigma^2_{qq}$  are the largest $q$ eigenvalues of the matrix $\vM\vM^\top$, and if $q<m$, the remaining eigenvalues of $\vM\vM^\top$ are all zeros. Then we have
	\begin{eqnarray*}
		\begin{aligned}
			\|\vX\vX^\top - \vM\vM^\top\|^2_F &\geq \sum_{i=1}^p \lambda^2_i + \sum_{i=1}^q \Sigma^4_{ii}-2\sum_{i=1}^p \lambda_i \Sigma^2_{ii}\\
			&=\sum_{i=1}^p(\lambda_i-\Sigma^2_{ii})^2 + \sum_{i=p+1}^q \Sigma^4_{ii}\geq \sum_{i=p+1}^q \Sigma^4_{ii},
		\end{aligned}
	\end{eqnarray*} where the equality is satisfied when we choose $\vX=\vU(:,1:p)\mbox{diag}(\Sigma_{11},\cdots,\Sigma_{pp})$. Let $\hat\Sigma=\mbox{diag}(\Sigma_{11},\cdots,\Sigma_{pp})$. The matrix $\hat\Sigma$ is invertible as the rank of $\vM$ is larger than $p$. Then for any orthogonal matrix $\vA\in \RR^{p\times p}$, $\hat\vX=\vU(:,1:p)\hat\Sigma\vA$ minimizes the objective function $\frac{1}{2}\|\vX\vX^\top - \vM\vM^\top\|^2_F$.
	
	After we find $\hat\vX=\vU(:,1:p)\hat\Sigma\vA$ for a certain orthogonal matrix $\vA$, we have 
	\begin{align*}
	\vY = \vM^\top\hat\vX(\hat\vX^\top\hat\vX)^{-1} = & \vV \Sigma\vU^\top\vU(:,1:p)\hat\Sigma\vA ((\vU(:,1:p)\hat\Sigma\vA)^\top\vU(:,1:p)\hat\Sigma\vA)^{-1} \\
	= &  \vV\Sigma\vU^\top\vU(:,1:p) \hat \Sigma^{-1}\vA \\
	= &  \vV(:,1:p)\hat \Sigma \hat \Sigma^{-1}\vA = \vV(:,1:p)\vA,
	\end{align*} where the third equality is due to the fact that 
	$$\Sigma\vU^\top\vU(:,1:p)=\left[ \begin{array}{l} \hat\Sigma_{p\times p} \\ \vzero_{(q-p)\times p} \end{array} \right].$$
	
	According to Theorem~\ref{thm:proximal}, $(\hat\vX,\vY)$ is an optimal solution of problem~\eqref{pro:model_fixedS} if $\mu=0$. Note that $\hat\vX=\vU(:,1:p)\hat\Sigma\vA$ is the thin SVD with $\hat\Sigma\in\RR^{p\times p}$. Then, the third step gives $\vX=\vU(:,1:p)\hat\Sigma_\lambda\vA$. Theorem~\ref{thm:proximal} shows that $(\vX,\vY)$ is an optimal solution of problem~\eqref{pro:model_fixedS}. 
\end{proof}

{\it Remark}: To find $\hat\vX$ in the first step, we apply the Gauss-Newton algorithm from~\cite{liu2015efficient}, which is previously used for RPCA in~\cite{Shen018}. The iteration is $\vX\leftarrow\vM\vM^\top\vX(\vX^\top\vX)^{-1}-\vX((\vX^\top\vX)^{-1}\vX^\top\vM\vM^\top\vX(\vX^\top\vX)^{-1}-\vI)/2$. When $p$ is small, computing the inverse of $\vX^\top\vX$ is fast. Though an iterative algorithm is required to solve this subproblem at each outer iteration, we can use the output from the previous outer iteration as the initial and the number of inner iterations is reduced significantly. Therefore, the computational time can be reduced significantly, as shown in Section~\ref{sec:num}. In the numerical experiments, the first Gauss-Newton algorithm requires several hundred iterations, while the number for following Gauss-Newton algorithms reduces to less than ten.

From Theorem~\ref{thm:proximal}, we say that we solve the proximal operator of the nonconvex function $\|\vL\|_*+\iota_{\mbox{rank}(\vL)\leq p}(\vL)$ exactly. Here the indicator function is defined as
\begin{align*}
\iota_{\mbox{rank}(\vL)\leq p}(\vL)=\left\{\begin{array}{ll}0, & \mbox{ if rank}(\vL)\leq p;\\ +\infty, & \mbox{ otherwise.}\end{array}\right.
\end{align*} With these theorems, we are ready to develop optimization algorithms for the general problem~\eqref{pro:general}. 

\subsection{Forward-backward}
First, we eliminate $\vS$, and it becomes the following problem with $\vL$ only:
\begin{align}\label{eq:flambda}
\begin{aligned}
&\Min_{\vL:\rank(\vL)\leq p} \min_{\vS} \frac{1}{2}\|\mathcal{A} (\vL) + \vS -\vD\|_F^2 + \lambda \|\vS\|_1 + \mu \|\vL\|_* \\
= &\Min_{\vL:\rank(\vL)\leq p} \min_{\vS} \left\{\frac{1}{2}\|\mathcal{A} (\vL) + \vS -\vD\|_F^2 + \lambda \|\vS\|_1\right\} + \mu \|\vL\|_* \\
= &\Min_{\vL:\rank(\vL)\leq p} f_\lambda(\vD-\mathcal{A}(\vL)) + \mu \|\vL\|_*.
\end{aligned}
\end{align}
Here $f_\lambda$ is the Moreau envelope of $\lambda|\cdot|$ defined by $f_\lambda(x)=\min_{y\in \RR} \{\lambda|y|+\frac{1}{2}(y-x)^2\}$. So it is differential and has a 1-Lipschitz continuous gradient. Then we can apply the proximal-gradient method (or forward-backward operator splitting). We take the gradient of $f_\lambda$, which is given by 
\begin{equation}\label{grad}
f'_\lambda(x) = x-\sign(x)\max(0,|x|-\lambda)=\sign(x)\min(\lambda,|x|).
\end{equation}
The forward-backward iteration for $\vL$ with stepsize $t$ is
\begin{equation}\label{forward-backwardL}
\vL^{k+1}=\prox_{t\mu} \left(\vL^k-t\mathcal{A}^* f'_\lambda(\mathcal{A}(\vL^k)-\vD)\right),
\end{equation}
where the proximal operator is defined by 
\begin{equation}\label{prox}
\prox_{\mu}(\vA)=\argmin_{\vL:\rank(\vL)\leq p} \frac{1}{2}\|\vL-\vA\|^2_F + \mu \|\vL\|_*.
\end{equation}
The algorithm is summarized in Alg.~\ref{Alg1}.

\begin{algorithm}[H]
	\SetAlgoLined
	\KwIn{$\vD$, $\mu$, $\lambda$, $p$, $\mathcal{A}$, stepsize $t$,  stopping criteria $\epsilon$, maximum number of iterations $Max\_Iter$, initialization $\vL^0 = \vzero$ }
	\KwOut{$\vL$, $\vS$}
	\For{k = 0, 1, 2, 3, \dots , Max\_Iter }{
		$\vS = \sign(\vD-\cA(\vL^k))\odot\max(0,|\vD-\cA(\vL^k)|-\lambda)$ \;
		$\vL^{k+1} = \prox_{t\mu}(\vL^k - t \mathcal{A}^* (\mathcal{A}(\vL^k)-\vD+\vS) $ using Gauss-Newton\; 
		\If{$\|\vL^{k+1}-\vL^{k}\|_F/\|\vL^k\|_F < \epsilon$}{\textbf{break}} 
	}
	\caption{Proposed Algorithm}\label{Alg1}
\end{algorithm}

{\bf Connection to~\cite{Shen018}.} Consider the special case with $\mathcal{A}=\mathcal{I}$ and $\mu=0$. We let $t=1$ in \eqref{forward-backwardL} and obtain the following iteration  
\begin{align*}
\vL^{k+1}=\prox_{0}(\vL^k-f'_\lambda(\vL^k-\vD))=\argmin_{\vL:\rank(\vL)\leq p}{1\over2}\|\vL+\vS^{k+1}-\vD\|^2,
\end{align*}
where $\vS^{k+1}=\sign(\vD-\vL^k)\odot\max(0,|\vD-\vL^k|-\lambda)$. This is exactly the algorithm in~\cite{Shen018} for solving~\eqref{pro:L1L1U}. 
It alternates between finding the best $\vS$ with $\vL$ fixed and the best $\vL$ (or $(\vX,\vY)$) with $\vS$ fixed.

Recently, the work~\cite{cai2019accelerated} proposed a novel RPCA algorithm with linear convergence. It projects matrices to special manifolds of low-rank matrices, and their truncated SVD can be computed efficiently.  Our matrix does not have this property in our algorithm, and a good initial guess from the previous iteration is necessary to reduce the computation in the Gauss-Newton method. 

\subsubsection{Convergence analysis} 
From the discussion above, problem \eqref{pro:general} can be solved by an iteration process of forward-backward splitting. In each iteration, we reduce the value of the objective function
\begin{equation}\label{targetfunction}
E(\vL,\vS)=\frac{1}{2}\|\mathcal{A}(\vL)+\vS-\vD\|^2_F+\lambda\|\vS\|_1+\mu\|\vL\|_*
\end{equation}
by applying proximal operators to $\vL$ and $\vS$ alternatively. The resulting iteration sequence $\{(\vL^k,\vS^k)\}_{k\geq 1}$ with some initial $(\vL^0,\vS^0)$ is explicitly given by
\begin{align}\label{sequences}
\begin{aligned}
\vS^k&=\sign(\vD-\mathcal{A}(\vL^{k-1}))\odot\max(0,|\vD-\mathcal{A}(\vL^{k-1})|-\lambda),\\
\vL^k
&=\prox_{t\mu} \left(\vL^{k-1}-t\mathcal{A}^*(\mathcal{A}(\vL^{k-1})+\vS^{k}-\vD)\right),\\
\end{aligned}
\end{align}  where the proximal operator $\prox_{t\mu}(\cdot)$ for updating $\vL$ is defined by \eqref{prox}. Here we use \eqref{grad} to derive 
\begin{align*}
&f'_{\lambda}(\mathcal{A}(\vL^{k-1})-\vD)\\
&=\mathcal{A}(\vL^{k-1})-\vD+\sign(\vD-\mathcal{A}(\vL^{k-1}))\odot\max(0,|\vD-\mathcal{A}(\vL^{k-1})|-\lambda)\\
&=\mathcal{A}(\vL^{k-1})+\vS^{k}-\vD.
\end{align*}  In this subsection, we  establish the convergence results for $\{(\vL^k,\vS^k)\}_{k\geq 1}$.  We will show that every limit point of $\{(\vL^k,\vS^k)\}_{k\geq 1}$, denoted by $(\vL^\star,\vS^\star)$, is a fixed point of the proximal operator, i.e., 
\begin{align}\label{fixedpoints}
\begin{aligned}
&\vS^\star=\sign(\vD-\mathcal{A}(\vL^\star))\odot\max(0,|\vD-\mathcal{A}(\vL^\star)|-\lambda),\\
&\vL^\star=\prox_{t\mu} \left(\vL^\star-t\mathcal{A}^*(\mathcal{A}(\vL^\star)+\vS^\star-\vD)\right).
\end{aligned}
\end{align} 
In practical execution, one can efficiently solve the proximal operator for $\vL$ by solving $(\vX^k,\vY^k)$ through
\begin{align}\label{solvingXY}
\begin{aligned}
\Min_{\vX,\vY }~&\frac{1}{2} \|\vX\vY^\top-\vL^{k-1}+t\mathcal{A}^*(\mathcal{A}(\vL^{k-1})+\vS^{k}-\vD)\|_F^2 + \mu\|\vX\|_*,\\
~\St~& \vY^\top\vY=\vI_{p\times p},
\end{aligned}
\end{align} 
and letting $\vL^k=\vX^k(\vY^k)^\top$. 
We also prove that if $(\vX^\star,\vY^\star,\vS^\star)$ is a limit point of $\{(\vX^k, \vY^k,\vS^k)\}_{k\geq 1}$, then $(\vX^\star(\vY^\star)^\top,\vS^\star)$ is a limit point of $\{(\vL^k,\vS^k)\}_{k\geq 1}$, and the limit point $(\vX^\star,\vY^\star,\vS^\star)$ is a stationary point of 
$$E(\vX\vY^\top,\vS)=\frac{1}{2}\|\mathcal{A}(\vX\vY^\top)+\vS-\vD\|^2_F+\lambda\|\vS\|_1+\mu\|\vX\vY^\top\|_*,$$ i.e., $(\vX^\star,\vY^\star,\vS^\star)$ satisfies the first-order optimality condition
\begin{align}\label{optimalitycondition}
\begin{aligned}
&\vzero \in [\mathcal{A}^*(\mathcal{A}(\vX^\star(\vY^\star)^\top)+\vS^\star-\vD)+\mu\partial \|\vX^\star(\vY^\star)^\top\|_* ]\vY^\star,\\
&\vzero \in (\vX^\star)^\top[\mathcal{A}^*(\mathcal{A}(\vX^\star(\vY^\star)^\top)+\vS^\star-\vD)+\mu\partial \|\vX^\star(\vY^\star)^\top\|_*],\\
&\vzero \in \mathcal{A}(\vX^\star(\vY^\star)^\top)+\vS^\star-\vD +\lambda \partial \|\vS^\star\|_1.\\
\end{aligned}
\end{align} We summarize these results in the following theorem.
\begin{theorem}\label{thm:convergence} 
	Define the objective function $E(\vL,\vS)$ as \eqref{targetfunction}. Let  $\{(\vL^k,\vS^k)\}_{k\geq 1}$ be a sequence generated by \eqref{sequences} with initial $(\vL^0,\vS^0)$ and stepsize $t<\frac{1}{\|\mathcal{A}\|^2}$, where $\vL^k=\vX^k(\vY^k)^\top$ with $(\vX^k, \vY^k)$ being solved from \eqref{solvingXY}. We have the following statements:
	\begin{enumerate}
		\item  The objective values $\{E(\vL^k,\vS^k)\}_{k\geq 1}$ are non-increasing along $\{(\vL^k,\vS^k)\}_{k\geq 1}$.
		\item The sequence $\{(\vL^k,\vS^k)\}_{k\geq 1}$ is bounded and thus has limit points.
		\item Every limit point $(\vL^\star,\vS^\star)$ of $\{(\vL^k,\vS^k)\}_{k\geq 1}$ satisfies \eqref{fixedpoints}.
		\item The sequence $\{(\vX^k,\vY^k,\vS^k)\}_{k\geq 1}$ is also bounded. In addition, for any limit point $(\vX^\star,\vY^\star,\vS^\star)$ of $\{(\vX^k,\vY^k,\vS^k)\}_{k\geq 1}$, $(\vX^\star(\vY^\star)^\top,\vS^\star)$ is a limit point of $\{(\vL^k,\vS^k)\}_{k\geq 1}$.
		\item Every limit point $(\vX^\star,\vY^\star,\vS^\star)$ of $\{(\vX^k,\vY^k,\vS^k)\}_{k\geq 1}$ is a stationary point of $E(\vX\vY^\top,\vS)$, which satisfies the first-order optimality condition in \eqref{optimalitycondition}.
	\end{enumerate} In addition, if $\mathcal{A}=\mathcal{I}$, we can take the stepsize $t=1$, and all the statements above still hold.
\end{theorem}
\begin{proof}
	We start by verifying the first two statements. For $k\geq 0$ and $t< \frac{1}{\|\mathcal{A}\|^2}$, we have
	\begin{align}\label{e17}
	\begin{aligned}
	&E(\vL^{k+1},\vS^{k+1})\\ 
	&=\frac{1}{2}\|\mathcal{A}(\vL^{k+1})-\mathcal{A}(\vL^k)\|^2_F+\langle \mathcal{A}(\vL^{k+1})-\mathcal{A}(\vL^k), \mathcal{A}(\vL^k)+\vS^{k+1}-\vD\rangle\\
	&\quad+\frac{1}{2}\|\mathcal{A}(\vL^k)+\vS^{k+1}-\vD\|^2_F+\lambda\|\vS^{k+1}\|_1+\mu\|\vL^{k+1}\|_*\\
	&\leq \frac{1}{2t}\|\vL^{k+1}-\vL^k\|^2_F + \langle \vL^{k+1}-\vL^k, \mathcal{A}^*f'_{\lambda}(\mathcal{A}(\vL^k)-\vD)\rangle+\mu\|\vL^{k+1}\|_*\\
	&\quad +\frac{1}{2}\|\mathcal{A}(\vL^k)+\vS^{k+1}-\vD\|^2_F+\lambda\|\vS^{k+1}\|_1
	+\left(\frac{\|\mathcal{A}\|^2}{2}-\frac{1}{2t}\right)\|\vL^{k+1}-\vL^k\|^2_F\\
	&=\frac{1}{t}\left\{\frac{1}{2}\|\vL^{k+1}-\vL^k+t\mathcal{A}^*f'_{\lambda}(\mathcal{A}(\vL^k)-\vD)\|^2_F+t\mu\|\vL^{k+1}\|_*\right\}\\
	&\quad-\frac{t}{2}\|\mathcal{A}^*f'_{\lambda}(\mathcal{A}(\vL^k)-\vD)\|^2_F+\left(\frac{\|\mathcal{A}\|^2}{2}-\frac{1}{2t}\right)\|\vL^{k+1}-\vL^k\|^2_F\\
	&\quad+\frac{1}{2}\|\mathcal{A}(\vL^k)+\vS^{k+1}-\vD\|^2_F+\lambda\|\vS^{k+1}\|_1,
	\end{aligned}
	\end{align} where the inequality is due to the facts that
	\begin{equation*}
	\|\mathcal{A}(\vL^{k+1})-\mathcal{A}(\vL^k)\|^2_F\leq \|\mathcal{A}\|^2\|\vL^{k+1}-\vL^k\|^2_F
	\end{equation*} and 
	\begin{equation*}
	\mathcal{A}(\vL^k)+\vS^{k+1}-\vD=f'_{\lambda}(\mathcal{A}(\vL^{k})-\vD).
	\end{equation*} Note that $\vL^{k+1}=\prox_{t\mu} \left(\vL^k-t\mathcal{A}^*
	f'_\lambda(\mathcal{A}(\vL^k)-\vD)\right)$, which solves 
	\begin{equation*}
	\Min_{\vL:\mbox{rank}(\vL)\leq p}~\frac{1}{2}\|\vL-\vL^k+t\mathcal{A}^*f'_{\lambda}(\mathcal{A}(\vL^k)-\vD)\|^2_F+t\mu\|\vL\|_*.
	\end{equation*} Since $\rank(\vL^{k}) \leq p$, we have
	\begin{align*}
	&\frac{1}{2}\|\vL^{k+1}-\vL^k+t\mathcal{A}^*f'_{\lambda}(\mathcal{A}(\vL^k)-\vD)\|^2_F+t\mu\|\vL^{k+1}\|_*\\
	&\leq \frac{1}{2}\|\vL^{k}-\vL^k+t\mathcal{A}^*f'_{\lambda}(\mathcal{A}(\vL^k)-\vD)\|^2_F+t\mu\|\vL^{k}\|_*\\
	&=\frac{t^2}{2}\|\mathcal{A}^*f'_{\lambda}(\mathcal{A}(\vL^k)-\vD)\|^2_F+t\mu \|\vL^k\|_*.
	\end{align*} 
	Substituting the above estimate to \eqref{e17} yields 
	\begin{align}\label{e18}
	\begin{aligned}
	&E(\vL^{k+1},\vS^{k+1})\leq \left(\frac{\|\mathcal{A}\|^2}{2}-\frac{1}{2t}\right)\|\vL^{k+1}-\vL^k\|^2_F\\
	&\quad +\frac{1}{2}\|\mathcal{A}(\vL^k)+\vS^{k+1}-\vD\|^2_F+\mu\|\vL^{k}\|_*+\lambda\|\vS^{k+1}\|_1.
	\end{aligned}
	\end{align}
	Moreover, we see that 
	\begin{equation*}
	\vS^{k+1}=\argmin_\vS~\frac{1}{2}\|\vS-(\vD-\mathcal{A}(\vL^k))\|^2_F+\lambda\|\vS\|_1.
	\end{equation*} Then from~\cite[Lemma 2]{lou2018fast}, there holds 
	\begin{align}\label{e19}
	\begin{aligned}
	&\frac{1}{2}\|\vS^{k+1}-(\vD-\mathcal{A}(\vL^k))\|^2_F+\lambda\|\vS^{k+1}\|_1\\
	&\leq \frac{1}{2}\|\vS^{k}-(\vD-\mathcal{A}(\vL^k))\|^2_F+\lambda\|\vS^{k}\|_1-\frac{1}{2}\|\vS^{k+1}-\vS^{k}\|^2_F.
	\end{aligned}
	\end{align} Combining estimates \eqref{e18} and \eqref{e19}, we find that
	\begin{equation}\label{e20}
	E(\vL^{k+1},\vS^{k+1})\leq E(\vL^{k},\vS^{k})+\left(\frac{\|\mathcal{A}\|^2}{2}-\frac{1}{2t}\right)\|\vL^{k+1}-\vL^k\|^2_F-\frac{1}{2}\|\vS^{k+1}-\vS^{k}\|^2_F.
	\end{equation} Since $\frac{\|\mathcal{A}\|^2}{2}-\frac{1}{2t}<0$, the estimate above implies $E(\vL^{k+1},\vS^{k+1})\leq E(\vL^{k},\vS^{k})$ for any $k\geq 0$, which verifies the first statement.  
	
	Note that the target function $E(\vL,\vS)$ is coercive, i.e., $E(\vL,\vS)\to +\infty$ when $\|\vL\|_F+\|\vS\|_F \to +\infty$. Since $E(\vL^{k},\vS^{k})\leq E(\vL^{0},\vS^{0})<+\infty, \forall k\geq 1$, this property guarantees that both $\{\vL^k\}_{k\geq 1}$ and $\{\vS^{k}\}_{k\geq 1}$ are bounded sequences, and thus the second statement holds. 
	
	For any limit point $(\vL^\star,\vS^\star)$ of $\{(\vL^k,\vS^k)\}_{k\geq 1}$, there exists a convergent subsequence $\{(\vL^{k_i},\vS^{k_i})\}_{i\geq 1}$ such that $\vL^{k_i}\to \vL^\star$ and $\vS^{k_i}\to \vS^\star$. On the other hand, we see that 
	\begin{align}\label{e21}
	\begin{aligned}
	\vS^{k_i+1}&=\sign(\vD-\mathcal{A}(\vL^{k_i}))\odot\max(0,|\vD-\mathcal{A}(\vL^{k_i})|-\lambda),\\
	\vL^{k_i+1}&=\prox_{t\mu} \left(\vL^{k_i}-t\mathcal{A}^*(\mathcal{A}(\vL^{k_i})+\vS^{k_i+1}-\vD)\right).
	\end{aligned}
	\end{align} 
	Summing both sides of \eqref{e20} from $k=0$ to $\infty$, we obtain
	\begin{equation*}
	\left(\frac{1}{t}-\|\mathcal{A}\|^2\right)\sum_{k=0}^{\infty}\|\vL^{k+1}-\vL^k\|^2_F+\sum_{k=0}^{\infty}\|\vS^{k+1}-\vS^{k}\|^2_F\leq 2 E(\vL^{0},\vS^{0}) <\infty.
	\end{equation*}
	This inequality guarantees that $\{\vS^{k_i+1}\}_{i\geq 1}$ has the same limit point $\vS^\star$ as that of $\{\vS^{k_i}\}_{i\geq 1}$, and $\{\vL^{k_i+1}\}_{i\geq 1}$ has the same limit point $\vL^\star$ as that of $\{\vL^{k_i}\}_{i \geq 1}$. Then by taking limits in both sides of the two equations in \eqref{e21}, we obtain the third statement. 
	
	Next we will prove the last two statements. As $\|\vX^k\|^2_F=\|\vL^k\|^2_F$ and $\|\vY^k\|^2_F=p$, we know that the sequence $\{(\vX^k,\vY^k,\vS^k)\}_{k\geq 1}$ is also bounded. Let $(\vX^\star,\vY^\star,\vS^\star)$ be a limit point of $\{(\vX^k,\vY^k,\vS^k)\}_{k\geq 1}$, which is the limitation of a subsequence $\{(\vX^{k_i},\vY^{k_i},\vS^{k_i})\}_{i\geq 1}$. Then we have
	$$\vL^{k_i}=\vX^{k_i}(\vY^{k_i})^\top \to \vX^\star(\vY^\star)^\top \mbox{ and } \vS^{k_i} \to \vS^\star,$$ i.e., $(\vX^\star(\vY^\star)^\top, \vS^\star)$ is the limit point of  $\{(\vL^k,\vS^k)\}_{k\geq 1}$ achieved by the subsequence $\{(\vL^{k_i},\vS^{k_i})\}_{i\geq 1}$. Thus the fourth statement is verified.
	
	Now we are in the position to prove the fifth statement. Due to the third and fourth statements, if $(\vX^\star,\vY^\star,\vS^\star)$ is a limit point of $\{(\vX^k,\vY^k,\vS^k)\}_{k\geq 1}$, i.e., $(\vX^\star(\vY^\star)^\top, \vS^\star)$ should satisfy \eqref{fixedpoints} 
	\begin{align}\label{e26}
	\begin{aligned}
	&\vS^\star=\sign(\vD-\mathcal{A}(\vX^\star(\vY^\star)^\top))\odot\max(0,|\vD-\mathcal{A}(\vX^\star(\vY^\star)^\top)|-\lambda),\\
	&\vX^\star(\vY^\star)^\top=\prox_{t\mu} \left(\vX^\star(\vY^\star)^\top-t\mathcal{A}^*(\mathcal{A}(\vX^\star(\vY^\star)^\top)+\vS^\star-\vD)\right).
	\end{aligned}
	\end{align} 
	The first condition in \eqref{e26} implies that the limit point $\vS^\star$ minimizes 
	$$\frac{1}{2}\|\mathcal{A}(\vX^\star(\vY^\star)^\top)+\vS-\vD\|^2_F+\lambda\|\vS\|_1+\mu\|\vX^\star(\vY^\star)^\top\|_*$$ over all $\vS \in \RR^{m\times n}$. Thus, $\vS^\star$ should satisfy the third condition in \eqref{optimalitycondition}.
	
	Moreover, since $\rank(\vX^\star(\vY^\star)^\top) \leq p$, the second condition in \eqref{e26} actually implies that $(\vX^\star,\vY^\star)$ is an optimal solution of the problem
	\begin{align*}
	&\Min_{\vX,\vY }~\frac{1}{2} \|\vX\vY^\top-\vX^\star(\vY^\star)^\top+t\mathcal{A}^*(\mathcal{A}(\vX^\star(\vY^\star)^\top)+\vS^\star-\vD)\|_F^2 + t\mu\|\vX\vY^\top\|_*. 
	\end{align*}Therefore, $(\vX^\star,\vY^\star)$ should satisfy the first-order optimality condition for $\vX$, which gives 
	\begin{align*}
	&[\vX^\star(\vY^\star)^\top-\vX^\star(\vY^\star)^\top+t\mathcal{A}^*(\mathcal{A}(\vX^\star(\vY^\star)^\top)+\vS^\star-\vD)]\vY^{\star} \\
	&+t\mu \partial \|\vX^\star(\vX^\star(\vY^\star)^\top)^\top\|_* \vY^\star\\
	=&t [\mathcal{A}^*(\mathcal{A}(\vX^\star(\vY^\star)^\top)+\vS^\star-\vD)+\mu\partial \|\vX^\star(\vY^\star)^\top\|_* ]\vY^\star\ni \vzero. 
	\end{align*}
	Similarly, from the first-order opitmality condition for $\vY$, one can verify that 
	$$\vzero \in (\vX^\star)^\top[\mathcal{A}^*(\mathcal{A}(\vX^\star(\vY^\star)^\top)+\vS^\star-\vD)+\mu\partial \|\vX^\star(\vY^\star)^\top\|_* ].$$ We thus derive the first two conditions in \eqref{optimalitycondition}.
	
	We will complete our proof by verifying the convergence results for the special case of $\mathcal{A}=\mathcal{I}$ and $t=1$. In this case, by the same method, one can derive a similar inequality as~\eqref{e20}, which is 
	\begin{equation*}
	E(\vL^{k+1},\vS^{k+1})\leq E(\vL^{k},\vS^{k})-\frac{1}{2}\|\vS^{k+1}-\vS^{k}\|^2_F.
	\end{equation*} Then $\{E(\vL^k,\vS^k)\}_{k\geq 1}$ are non-increasing along $\{(\vL^k,\vS^k)\}_{k\geq 1}$, and $\{(\vL^k,\vS^k)\}_{k\geq 1}$ is bounded due to the coerciveness of $E(\vL,\vS)$. Let $(\vL^\star, \vS^\star)$ be the limit point of  $\{(\vL^k,\vS^k)\}_{k\geq 1}$ achieved by the subsequence $\{(\vL^{k_i},\vS^{k_i})\}_{i\geq 1}$. Recall the iterations for updating  $\vS^{k_i+1}$ and $\vL^{k_i}$ given by
	\begin{align}\label{updating}
	\begin{aligned}
	\vS^{k_i+1}&=\sign(\vD-\vL^{k_i})\odot\max(0,|\vD-\vL^{k_i}|-\lambda),\\
	\vL^{k_i}&=\prox_{\mu} \left(\vD-\vS^{k_i}\right).
	\end{aligned}
	\end{align} Since $\sum_{k=0}^{\infty}\|\vS^{k+1}-\vS^{k}\|^2_F\leq 2 E(\vL^{0},\vS^{0}) <+\infty$, $\{\vS^{k_i+1}\}_{i\geq 1}$ has the same limit point $\vS^\star$ as that of $\{\vS^{k_i}\}_{i\geq 1}$. Taking limits in both sides of equations \eqref{updating} yields the condition~\eqref{fixedpoints} for $\mathcal{A}=\mathcal{I}$ and $t=1$. The last two statements can be verified by exactly the same arguments for the general case. We thus complete the proof.
\end{proof}

\subsection{An accelerated algorithm}
We show in the previous subsection that Alg.~\ref{Alg1} is a forward-backward splitting or proximal gradient algorithm for a nonconvex problem. Recently, accelerated proximal gradient (APG) algorithms are proposed for nonconvex problems to reduce the computational time without sacrificing convergence~\cite{li2015global,li2015accelerated}. In this paper, we adopt the nonmonotone APG~\cite[Alg. 2]{li2015accelerated} because of its better performance shown in~\cite{li2015accelerated}. The algorithm is described in Alg.~\ref{Alg2}. We let $\delta=1$ and $\eta=0.6$ in the numerical experiments.

\begin{algorithm}[H]
	\SetAlgoLined
	\SetAlgoLined
	\KwIn{$\vD$, $\mu$, $\lambda$, $p$, $\mathcal{A}$, stepsize $t$, $\eta \in [0,1)$, $\delta > 0$, stopping criteria $\epsilon$, maximum number of iterations $Max\_Iter$, initialization: $\vL^0 = \vL^1 = \vZ^1 = \textbf{0}$, $t^0 = 0$, $t^1 = q^1=1$,  $c^1 = F(\vL^1)$}
	\KwOut{$\vL$, $\vS$}
	\For{k = 1, 2, 3, .., Max\_Iter}{
		$\vL = \vL^k + \frac{t^{k-1}}{t^k}(\vZ^k - \vL^k) +\frac{t^{k-1}-1}{t^k}(\vL^k - \vL^{k-1})$\; 
		$\vS = \sign(\vD-\cA(\vL))\odot\max(0,|\vD-\cA(\vL)|-\lambda)$\;
		$\vZ^{k+1} = \prox_{t\mu}(\vL - t\mathcal{A}^*(\mathcal{A}(\vL) - \vD + \vS))$\;
		\eIf{$F(\vZ^{k+1}) \leq c^k - \delta \|\vZ^{k+1} - \vL\|^2$}{
			$\vL^{k+1} = \vZ^{k+1}$\;
		}{
			$\vS^k = \sign(\vD-\cA(\vL^k))\odot\max(0,|\vD-\cA(\vL^k)|-\lambda)$\;
			$\vV^{k+1} = \prox_{t\mu}(\vL^k - t\mathcal{A}^*(\mathcal{A}(\vL^k) - \vD + \vS^k))$\;
			$\vL^{k+1}=
			\begin{cases}
			\vZ^{k+1}& \textbf{if}~ F(\vZ^{k+1}) \leq F(\vV^{k+1});\\
			\vV^{k+1}& \textbf{otherwise};
			\end{cases}$\
		}
		\If{$\|\vL^{k}-\vL^{k-1}\|_F/\|\vL^{k-1}\|_F < \epsilon$}{\textbf{break}} 
		$t^{k+1} = \frac{\sqrt{4(t^k)^2+1} + 1}{2}$\;
		$q^{k+1} = \eta q^k + 1$\;
		$c^{k+1} = \frac{\eta q^k c^k + F(\vL^{k+1})}{q^{k+1}}$\;
	}
	\caption{Accelerated algorithm with nonmonotone APG}\label{Alg2}
\end{algorithm}

\section{Numerical experiments}\label{sec:num}
In this section, we use synthetic data and real images to demonstrate the performance of our proposed model and algorithms. The code to reproduce the results in this section can be found at \url{https://github.com/mingyan08/RPCA_Rank_Bound}.

\subsection{Synthetic data}
We would like to recover the low-rank matrix from a noisy matrix that is contaminated by a sparse matrix and Gaussian noise. We create a true low-rank $500 \times 500$ matrix $\vL^\star$ by multiplying a random $500\times r$ matrix and a random $r\times 500$ matrix, where their components are generated from standard normal distribution independently. We calculate the mean of the absolute values of all the components in $\vL^\star$ and denote it as $c$. Then we randomly select $s\%$ of the components and replace their values with uniformly distributed random values from $[-3c, 3c]$. After that, we add small Gaussian noise $\mathcal{N} (0, \sigma^2)$ to all components of the matrix. We let $t=1.7$ in the experiments because of fast convergence, though the convergence results in Theorem~\ref{thm:convergence} require $t<1$.

\subsubsection{Low-rank matrix recovery}
We fix $\sigma = 0.05$ for the Gaussian noise and set the upper bound of the rank to be $p=r+5$. We stop all algorithms when the relative error at the $k$-th iteration, which is defined as
\begin{equation*}
RE(\vL^{k+1},\vL^k) := \frac{\|\vL^{k+1} - \vL^k\|_F}{\|\vL^k\|_F},
\end{equation*}
is less than $10^{-4}$. We use the relative error to $\vL^\star$, which is defined as
\begin{equation*}
RE(\vL,\vL^*) := \frac{\|\vL - \vL^\star\|_F}{\|\vL^\star\|_F},
\end{equation*}
to evaluate the performance of our proposed model and that in~\cite{Shen018}. First, we consider the case with $r=25$ and $s=20$. We plot a contour map of the relative error to $\vL^\star$ for different parameters $\mu$ and $\lambda$ in Fig.~\ref{fig:contour}. From this contour map, we can see that the best parameter does not happen when $\mu=0$, which corresponds to the model in~\cite{Shen018}. It verifies the better performance of our proposed model with appropriate parameters. In this subsection, we set $\lambda=0.02$ for Shen et al.'s and $(\mu=0.6,~\lambda=0.04)$ for our proposed algorithms. 

\begin{figure}[!ht]
	\includegraphics[width=0.7\textwidth]{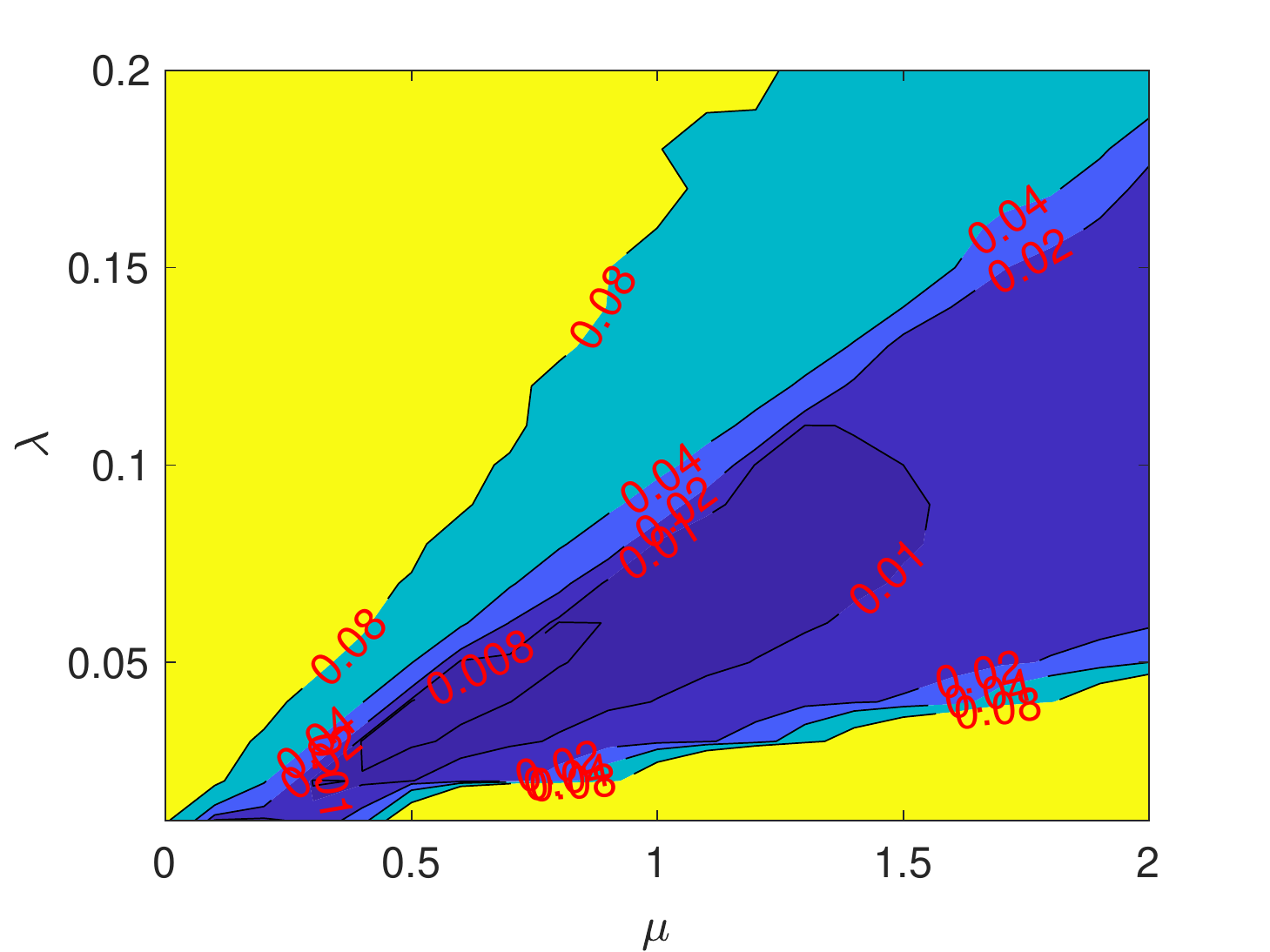}
	\caption{The contour map of the relative error to $\vL^\star$ for different parameters. In this experiment, we set $r=25$ and $s=20$. The upper bound of the rank is set to be $p=30$.}\label{fig:contour}
\end{figure}

In addition, we consider another two settings for $(r,s)$, and the comparison with different algorithms is shown in Table~\ref{tab}. In this table, we also compare the number of iterations for three algorithms: Shen et al.'s, Alg.~\ref{Alg1}, and Alg.~\ref{Alg2}. From this table, we can see that both Alg.~\ref{Alg1} and Alg.~\ref{Alg2} have better performance and fewer iterations than~\cite{Shen018}. The accelerated Alg.~\ref{Alg2} has the fewest iterations, but its performance in terms of $RE(\vL,\vL^\star)$ is not as good as Alg.~\ref{Alg1} for the last case. It is because we stop both algorithms when the stopping criteria is satisfied, and the algorithms are not converged yet. We checked the objective function values for both algorithms, and the value for Alg.~\ref{Alg2} is smaller than that for Alg.~\ref{Alg1} in this case. Therefore, if we want a solution close to the true low-rank matrix $\vL^\star$, we may need to stop early before the convergence, which is the same as many models for inverse problems.

\begin{table}[!ht]
	\begin{center}
		\begin{tabular}{|l|c|c|c|c|c|c|c|c|}
			\hline
			\multirow{2}{*}{$r$} &\multirow{2}{*}{s} & \multicolumn{2}{c|}{Shen et al.'s \cite{Shen018}}& \multicolumn{2}{c|}{Alg. 1}& \multicolumn{2}{c|}{Alg.2} \\\cline{3-8}
			&  & $RE(\vL,\vL^\star)$  & $\#$ iter & $RE(\vL,\vL^\star)$ & $\#$ iter & $RE(\vL,\vL^\star)$ & $\#$ iter\\
			\hline
			25 & 20 & 0.0745 & 1318 & 0.0075 & 296 & 0.0075 & 68\\
			\hline
			50  & 20 & 0.0496 & 1434 & 0.0101 & 473 & 0.0088 & 77\\
			\hline
			25 & 40 & 0.0990 & 2443 & 0.0635 & 796 & 0.0915 & 187\\
			\hline
		\end{tabular}
	\end{center}
	\caption{Comparison of three RPCA algorithms. We compare the relative error of their solutions to the true low-rank matrix and the number of iterations. Both Alg.~\ref{Alg1} and Alg.~\ref{Alg2} have better performance than~\cite{Shen018} in terms of the relative error and the number of iterations. Alg.~\ref{Alg2} has the fewest iterations but the relative error could be large. It is because the true low-rank matrix is not the optimal solution to the optimization problem, and the trajectory of the iterations moves close to $\vL^\star$ before it approaches the optimal solution.}\label{tab}
\end{table}

\subsubsection{Robustness of the model}\label{sec:robust}
In this experiment, we compare the robustness of our proposed model with that of~\cite{Shen018}. We let $r=25$ and $s=20$. Then we run both models for $p$ from 15 to 35. The comparison of the relative error to $\vL^\star$ is shown in Fig.~\ref{fig:robust}. We let $\lambda=0.02$ for Shen et al.'s and ($\mu=0.6,~\lambda=0.04$) for Alg.~\ref{Alg2}. It shows that our proposed model is robust to the parameter $p$, as long as it is not smaller than the true rank $r$. 
\begin{figure}[!ht]
	\centering
	\includegraphics[width=0.7\textwidth]{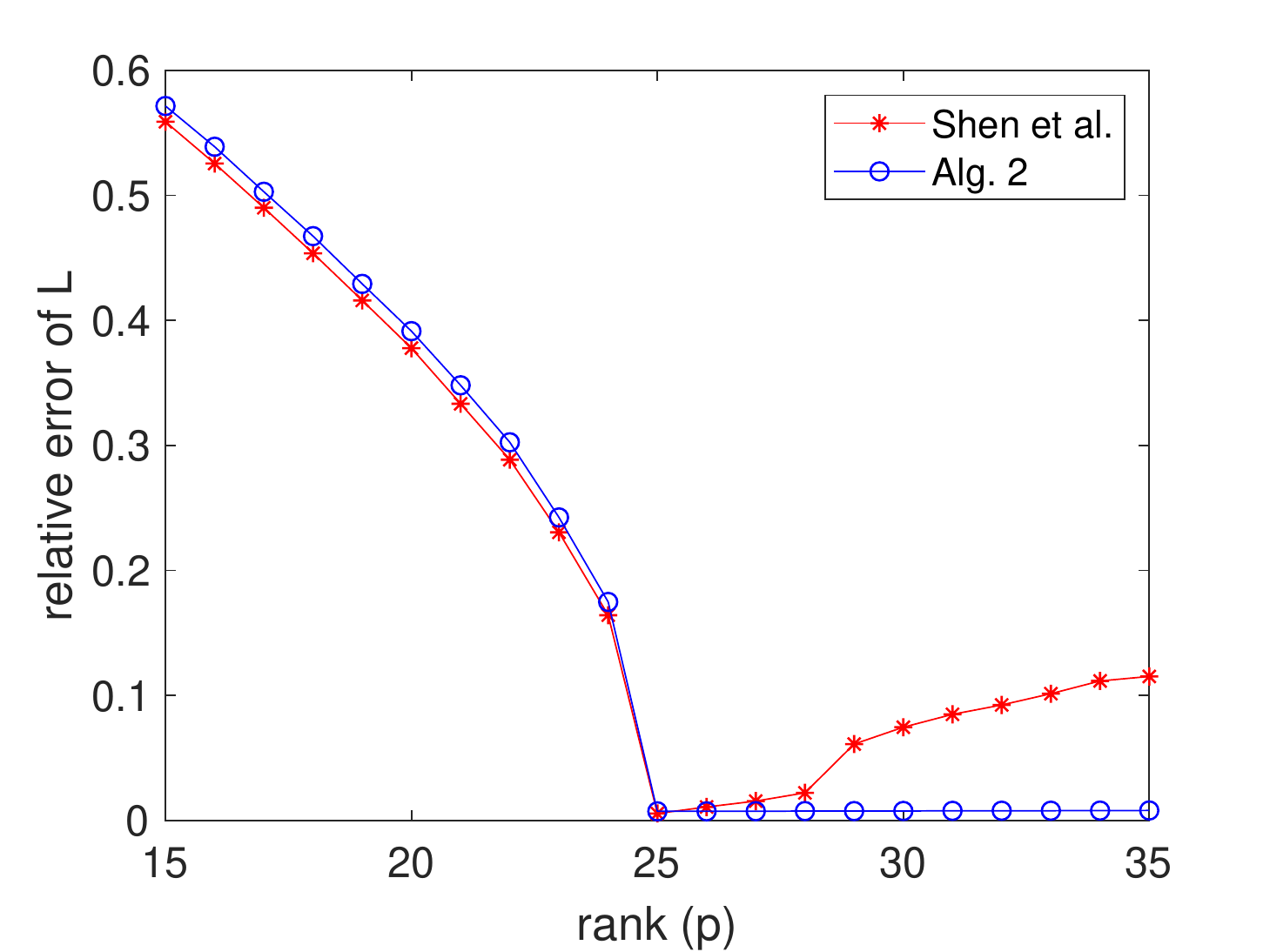}
	\caption{The relative error to the true low-rank matrix vs the rank $p$ for Shen et al.'s and Alg.~\ref{Alg2}. Alg.~\ref{Alg2} is robust to $p$, as long as $p$ is not smaller than the true rank 25.} \label{fig:robust}
\end{figure}

\subsubsection{Low-rank matrix recovery with missing entries} 
In this experiment, we try to recover the low-rank matrix when there are missing entries in the matrix. Therefore, the operator $\mathcal{A}$ is not the identity $\mathcal{I}$. We randomly select the missing entries from all the entries. We let $r=25$ and add both the sparse noise with parameter $s$ and the Gaussian noise with parameter $\sigma$ to the true matrix $\vL^\star$. Then we apply Alg.~\ref{Alg2} to recover the low-rank matrix, and the relative error to $\vL^\star$ is used to evaluate the performance. The results for different settings are in Table~\ref{tab2}. For the first three cases with $s=20$, we choose $(\mu=0.5,~\lambda=0.04)$, while we let $(\mu=0.1,~\lambda=0.01)$ for the last case with $s=5$. Note that, even with missing entries, Alg.~\ref{Alg2} can reconstruct the low-rank matrix accurately.

\begin{table}[!ht]
	\begin{center}
		\begin{tabular}{|c|r|c|c|}
			\hline
			{s} & {$\sigma$} & ratio of missing entries &$RE(\vL,\vL^\star)$ by Alg.~\ref{Alg2} \\
			\hline
			20 & 0.05 & 10\%& 0.0079 \\
			\hline
			20 & 0.05 & 20\% & 0.0088\\
			\hline
			20 & 0.05  & 50\%& 0.0201\\
			\hline
			5 & 0.01  & 50\% & 0.0015\\
			\hline
		\end{tabular}
	\end{center}
	\caption{Performance of Alg.~\ref{Alg2} on low-rank matrix recovery with missing entries. We change the level of sparsity in the sparse noise, standard deviation of the Gaussian noise, and the ratio of missing entries. }\label{tab2}
\end{table}

\subsection{Real image experiment}\label{sec:image}
In this section, we consider the three algorithms applied to image processing problems. Since natural images are not low-rank essentially, we consider two cases on two different images (`cameraman' and `Barbara'). For the $256\times 256$ cameraman image (the pixel values are from 0 to 255), we create an image with rank 37 from a low-rank approximation of the original image. Then we add $20\%$ salt and pepper impulse noise and Gaussian noise with standard variance 4. We set 42 as the upper bound of the rank of the low-rank image for all algorithms. We let $\lambda=0.03$ for Shen et al. and $(\mu=0.5,~\lambda = 0.06)$ for our model. To compare the performance of both models, we use the relative error defined in the last subsection and peak signal to noise ratio (PSNR) defined as 
$$\mbox{PSNR} := 10 \log_{10} {\mbox{Peak\_Val}^2\over \mbox{MSE}}.$$
Here $\mbox{Peak\_Val}$ is the largest value allowed at a pixel (255 in our case), and $\mbox{MSE}$ is the mean squared error between the recovered image and the true image. The numerical results are shown in Fig.~\ref{fig:camera}. From Fig.~\ref{fig:camera}(A-C), we can see that our proposed model performs better than Shen et al.~\cite{Shen018}. For the proposed model, we also compare the speed of three algorithms: Alg.~\ref{Alg1}, Alg.~\ref{Alg1} with standard SVD, and Alg.~\ref{Alg2} in Fig.~\ref{fig:camera}(D).  For both plots, we can see that the Gauss-Newton approach increases the speed comparing to the standard SVD approach. From the decrease of the objective function value, we can see that the accelerated algorithm Alg.~\ref{Alg2} is faster than the nonaccelerated Alg.~\ref{Alg1}. 

\begin{figure}[!ht]
	\centering
	\begin{subfigure}[t]{0.4\textwidth}
		\includegraphics[trim={54pt 48pt 54pt 18pt},clip,width=\linewidth]{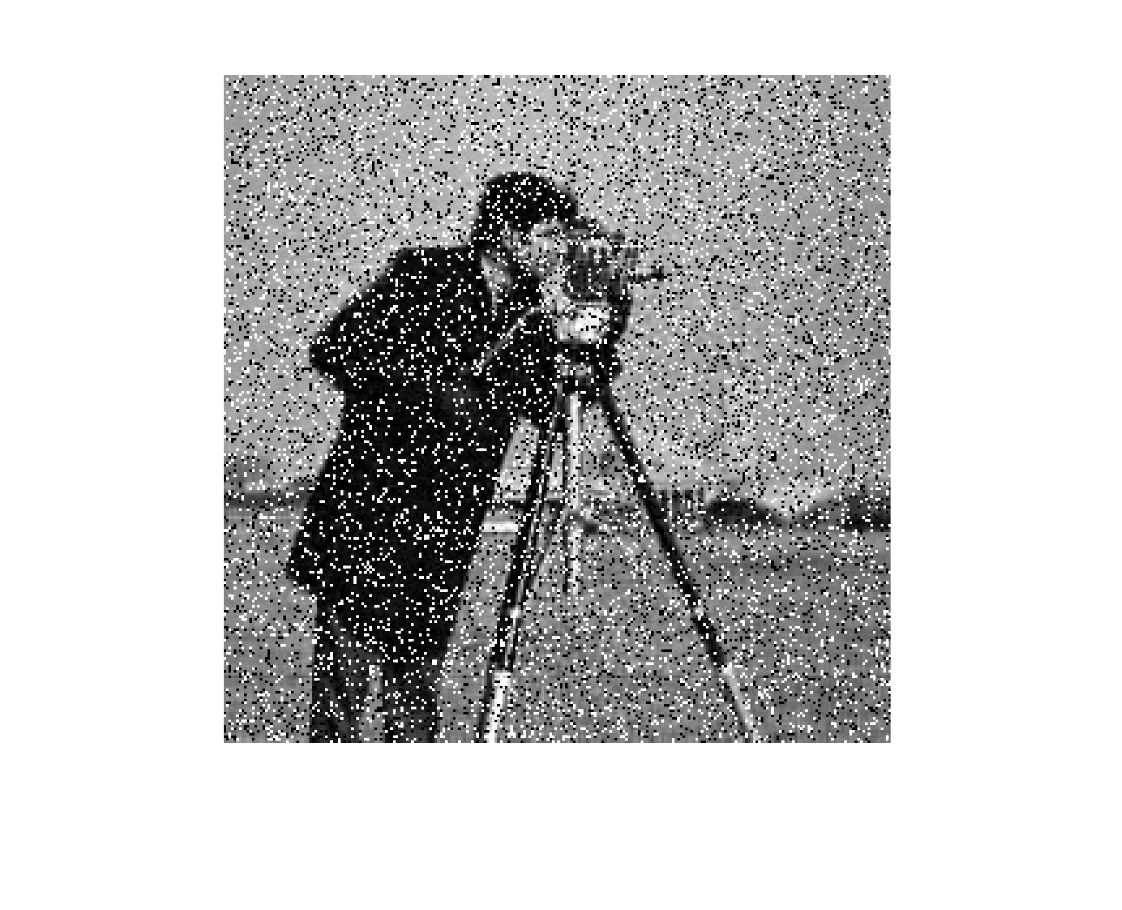}
		\caption{Corrupted image \\RE: 0.4760, PSNR:  12.76}
	\end{subfigure}\quad \quad
	\begin{subfigure}[t]{0.4\textwidth}
		\includegraphics[trim={54pt 48pt 54pt 18pt},clip, width=\linewidth]{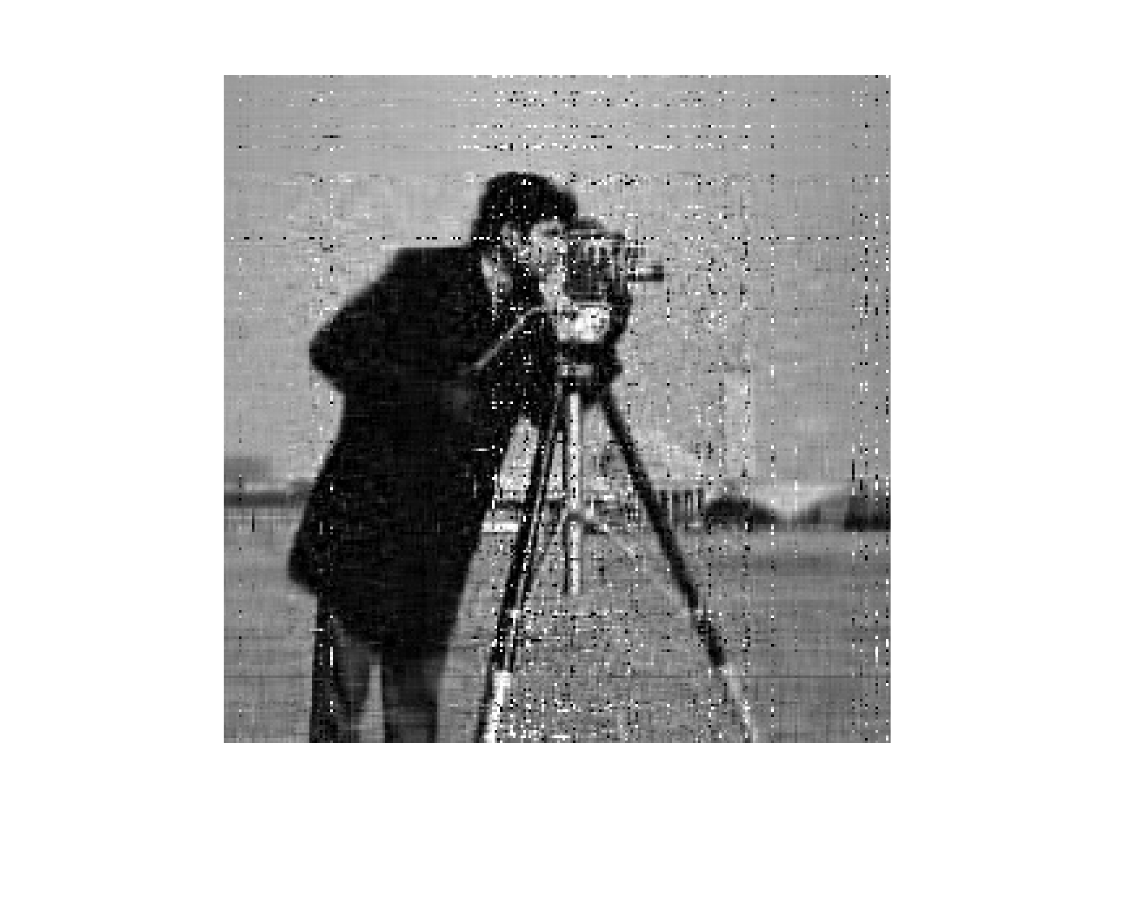}
		\caption{Recovered by Shen et al.\\RE: 0.1736, PSNR: 21.52}
	\end{subfigure}\\
	\begin{subfigure}[t]{0.4\textwidth}
		\includegraphics[trim={54pt 48pt 54pt 18pt},clip,width=\linewidth]{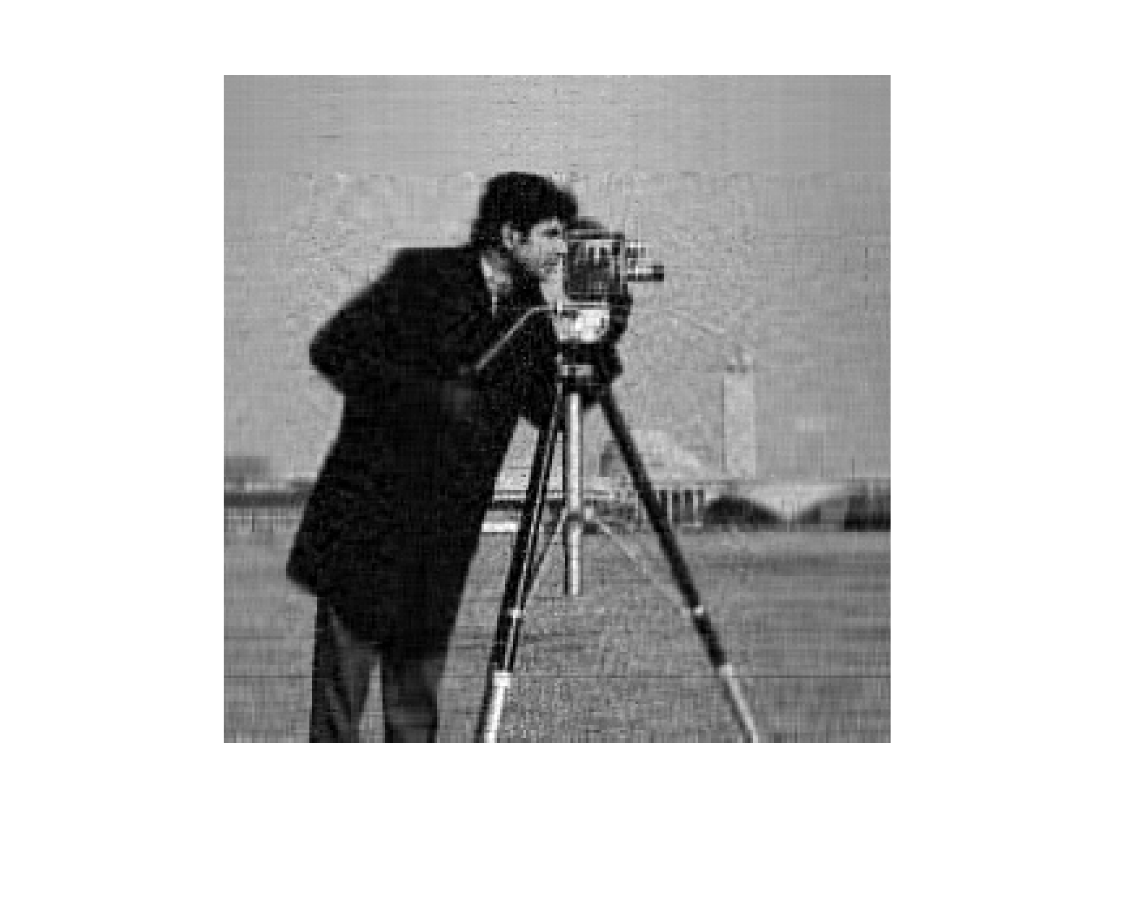}
		\caption{Recovered by Alg.~\ref{Alg2} \\RE: 0.0457, PSNR:33.11}
	\end{subfigure}
	\begin{subfigure}[t]{0.45\textwidth}
		\includegraphics[width=\linewidth]{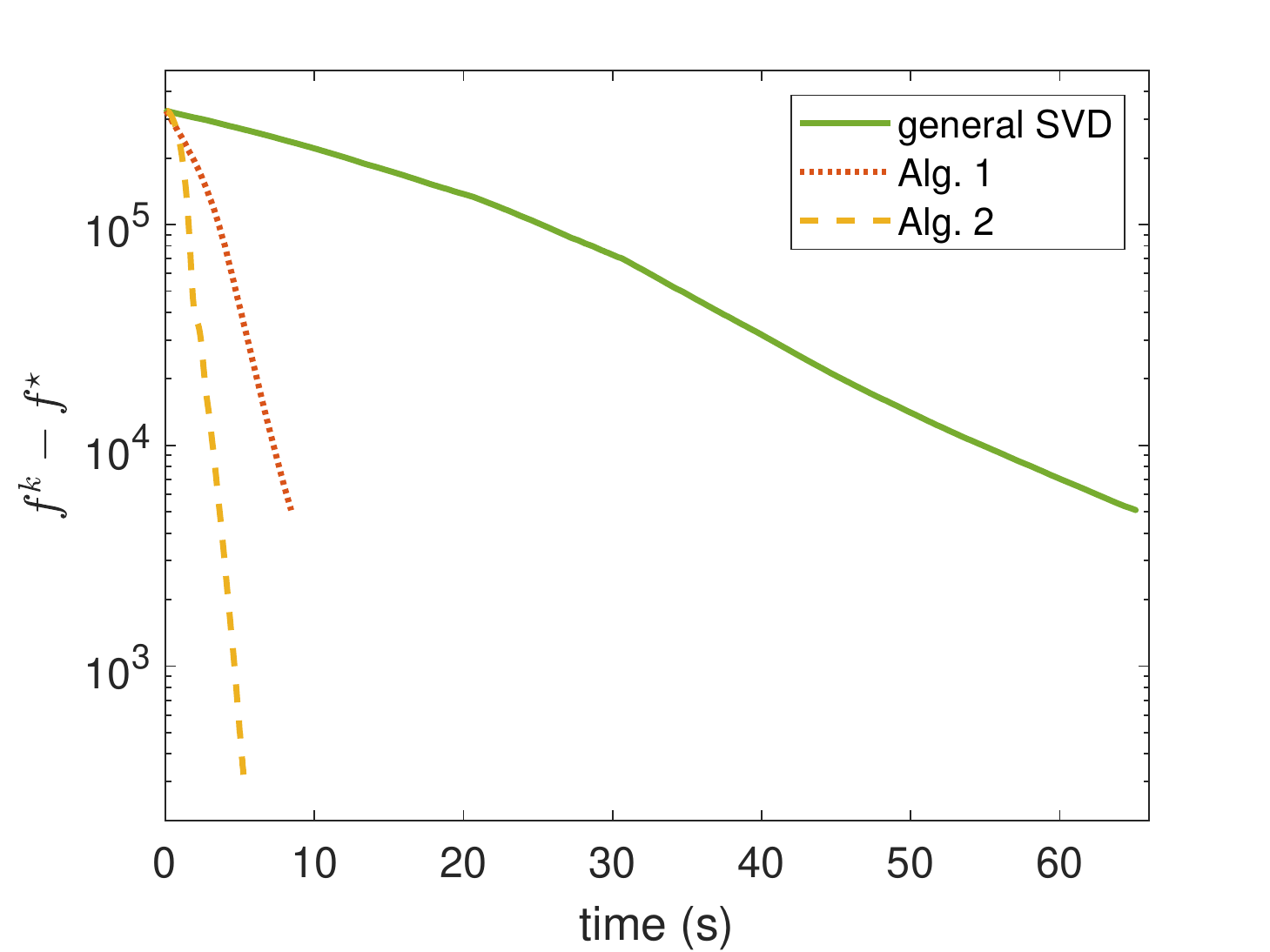}
		\caption{Comparison of the objective function value vs time for three algorithms}
	\end{subfigure}
	\caption{The numerical experiment on the `cameraman' image. (A-C) show that the proposed model performs better than Shen et al.'s both visually and in terms of RE and PSNR. (D) compares the objective values vs time for general SVD, Alg.~\ref{Alg1}, and Alg.~\ref{Alg2}. Here $f^\star$ is the value obtained by Alg.~\ref{Alg2} with more iterations. It shows the fast speed with the Gauss-Newton approach and acceleration. With the Gauss-Newton approach, the computation time for Alg.~\ref{Alg1} is reduced to about 1/7 of the one with standard SVD (from 65.11s to 8.43s). The accelerated Alg.~\ref{Alg2} requires 5.2s, though the number of iterations is reduced from 3194 to 360.}\label{fig:camera}
\end{figure}

Next, we use the original $512\times 512$ barbara image (the pixel values are from 0 to 255) without modification and add the same two types of noise as in the cameraman image. Because the original image is not low-rank, we choose the upper bound of rank $p=50$. We let $\lambda=0.03$ for Shen et al. and $(\mu=0.5,~\lambda = 0.06)$ for our model. The comparison result is shown in Fig.~\ref{fig:barbara}, and it is similar to the cameraman image. We also applied the acceleration to Shen et al.'s algorithm and obtained a better image with $\mbox{RE}=0.1447$ and $\mbox{PSNR}=22.37$.

\begin{figure}[!ht]
	\centering
	\begin{subfigure}[t]{0.4\textwidth}
		\includegraphics[trim={54pt 48pt 54pt 18pt},clip,width=\linewidth]{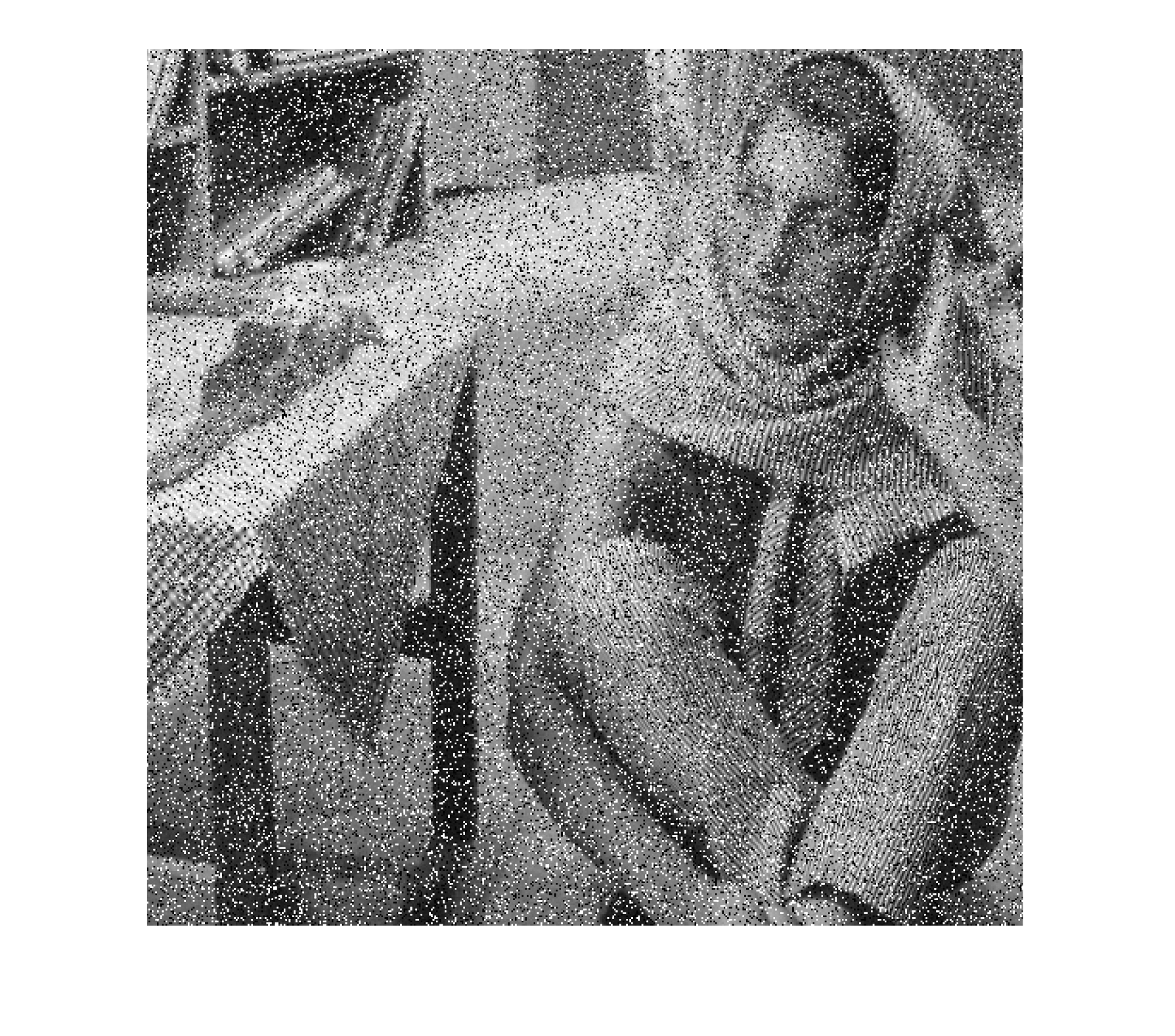}
		\caption{Corrupted image \\RE: 0.4821, PSNR:  11.91}
	\end{subfigure}\quad\quad
	\begin{subfigure}[t]{0.4\textwidth}
		\includegraphics[trim={54pt 48pt 54pt 18pt},clip, width=\linewidth]{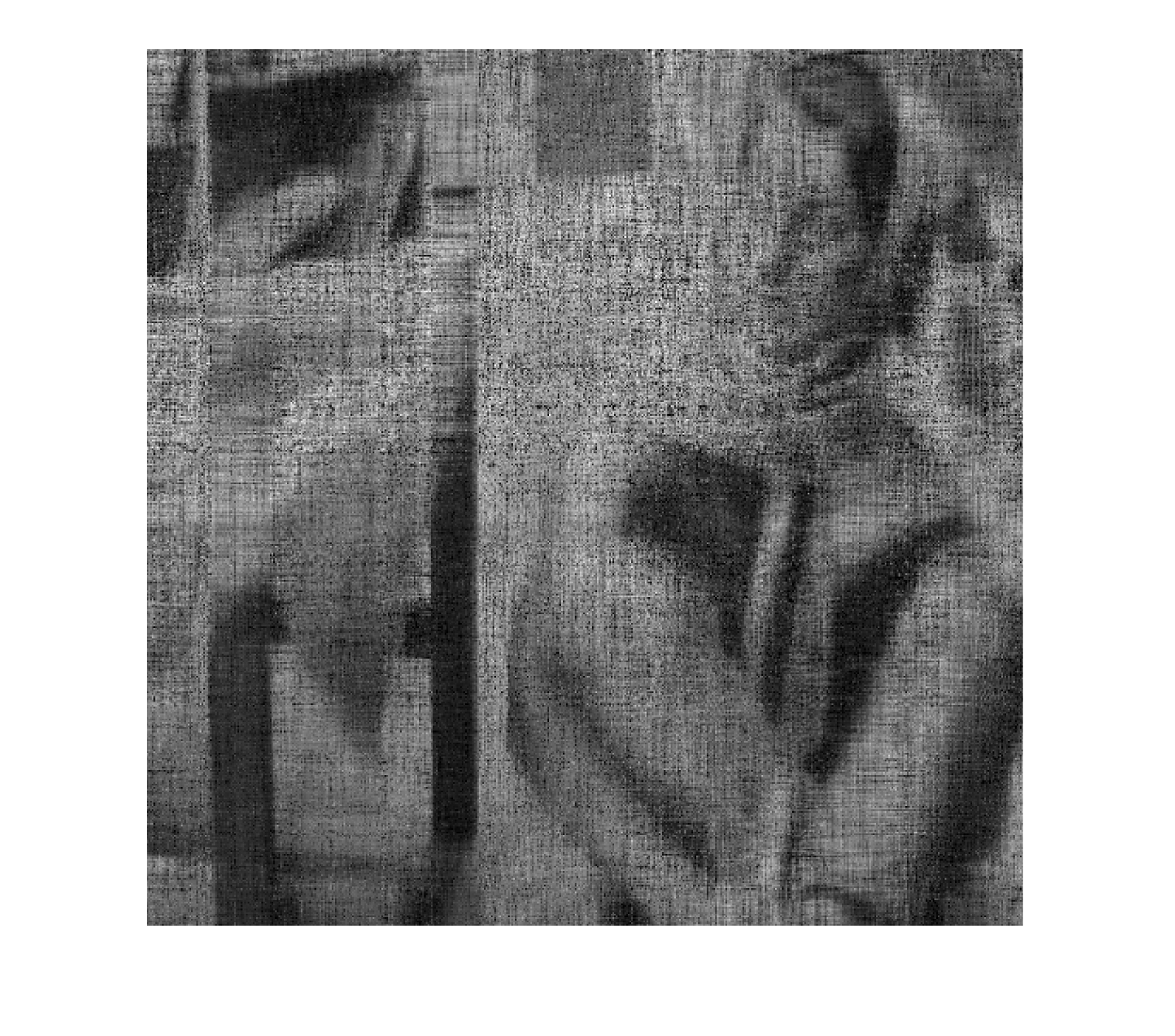}
		\caption{Recovered by Shen et al\\RE: 0.3368, PSNR: 15.03}
	\end{subfigure}\\
	\begin{subfigure}[t]{0.4\textwidth}
		\includegraphics[trim={54pt 48pt 54pt 18pt},clip,width=\linewidth]{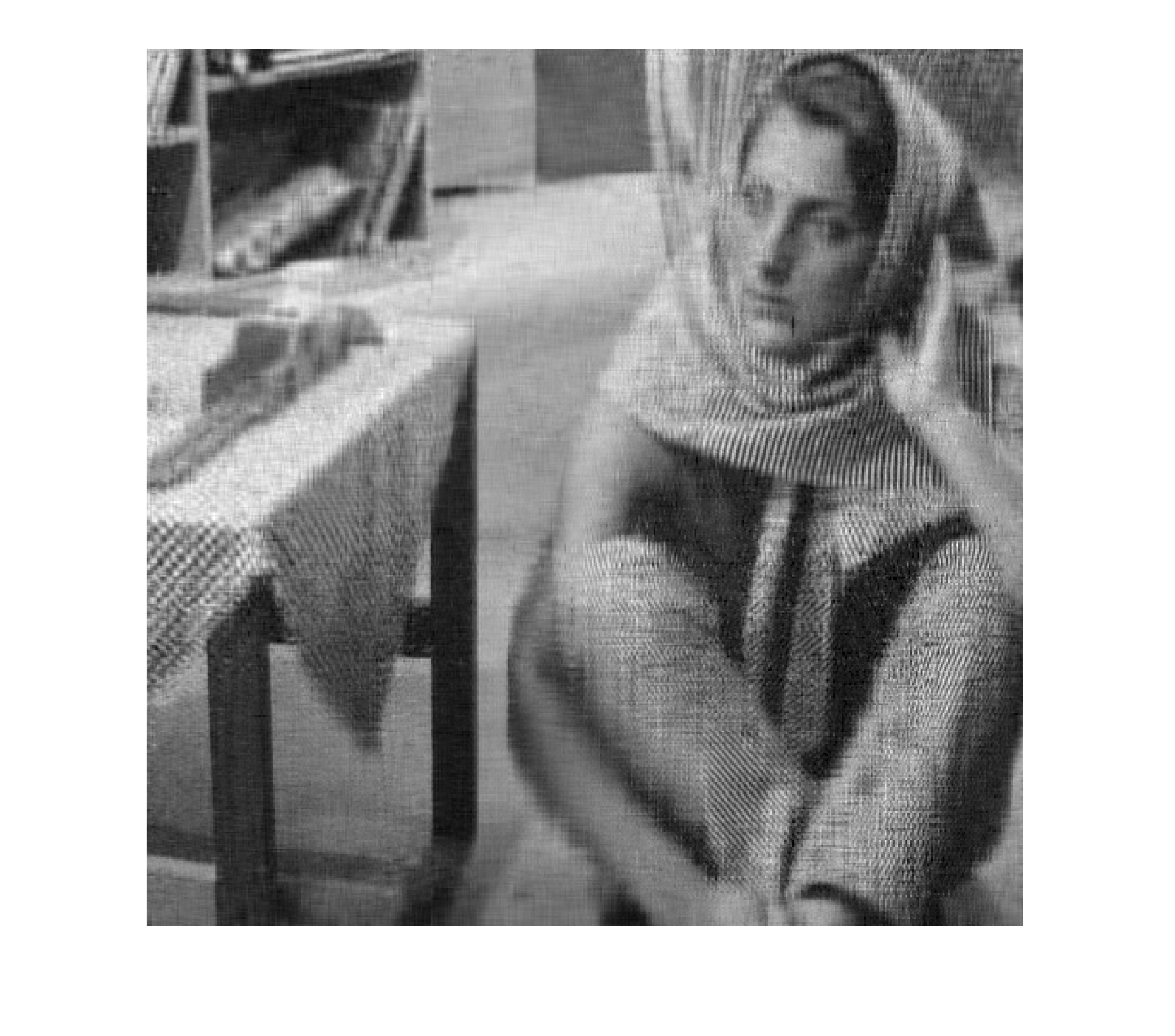}
		\caption{Recovered by Alg.~\ref{Alg2} \\RE: 0.1317, PSNR: 23.18}
	\end{subfigure}
	\begin{subfigure}[t]{0.45\textwidth}
		\includegraphics[width=\linewidth]{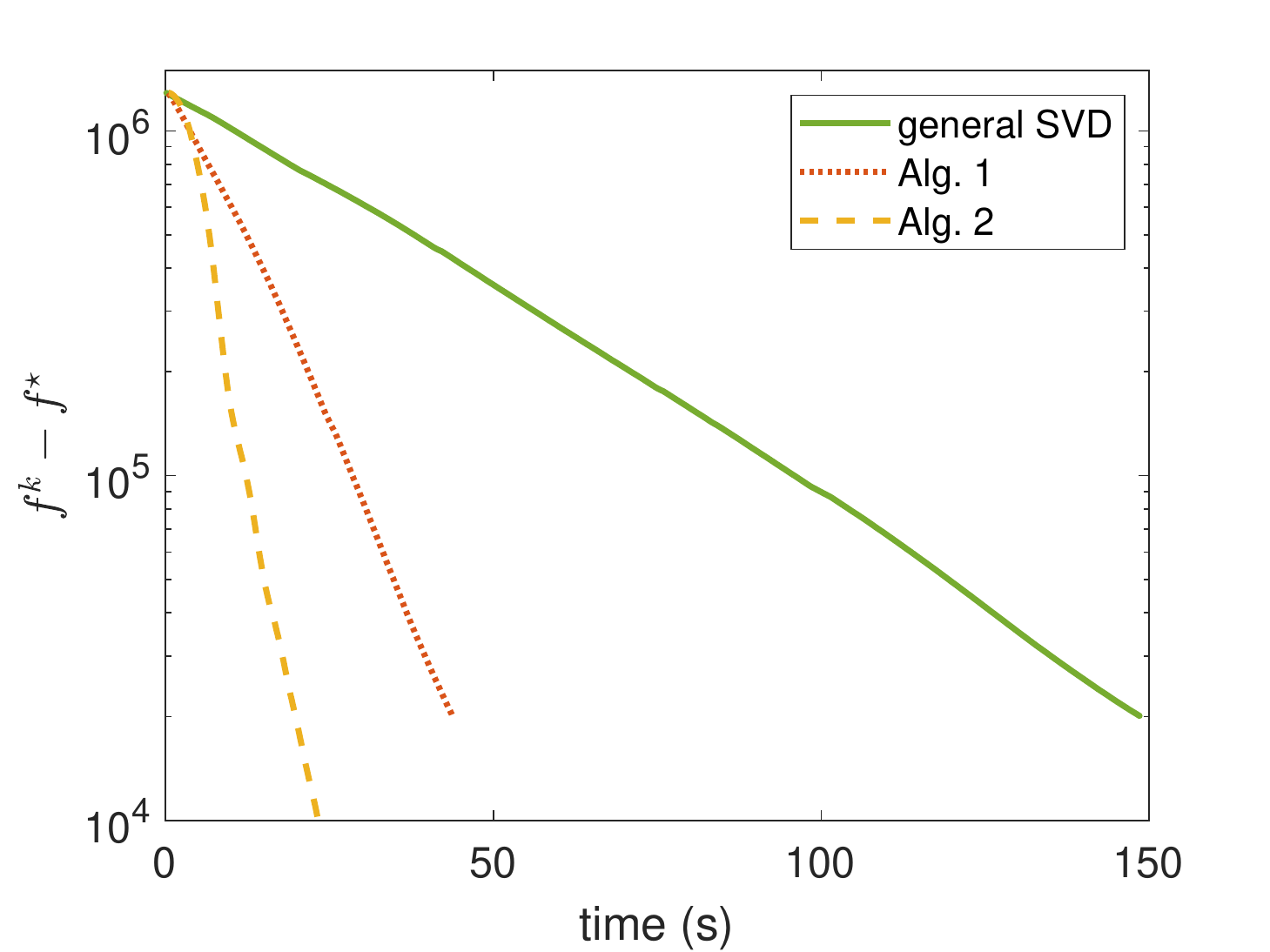}
		\caption{Comparison of the objective function value vs time for three algorithms}
	\end{subfigure}
	\caption{The numerical experiment on the `Barbara' image. (A-C) show that the proposed model performs better than Shen et al.'s both visually and in terms of RE and PSNR.  (D) compares the objective values vs time for general SVD, Alg.~\ref{Alg1}, and Alg.~\ref{Alg2}. Here $f^\star$ is the value obtained by Alg.~\ref{Alg2} with more iterations. It shows the fast speed with the Gauss-Newton approach and acceleration. With the Gauss-Newton approach, the computation time for Alg.~\ref{Alg1} is reduced to less than 1/3 of the one with standard SVD (from 148.6s to 43.7s). The accelerated Alg.~\ref{Alg2} requires 23.3s, though the number of iterations is reduced from 3210 to 300.}\label{fig:barbara}
\end{figure}

\section{Concluding remarks}\label{section4}
In this paper, we introduced a new model for RPCA when an upper bound of the rank is provided. For the unconstrained RPCA problem, we formulate it as the sum of one smooth function and one nonsmooth nonconvex function. Then we derive an algorithm based on proximal-gradient. This proposed algorithm has the alternating minimization algorithm~\cite{Shen018} as a special case. Because of the connection between this algorithm and proximal gradient, we adopted an acceleration approach and proposed an accelerated algorithm. Both proposed algorithms have two advantages comparing to existing algorithms. First, different from algorithms that require accurate rank estimations, the proposed algorithms are robust to the upper bound of the rank. Second, we apply the Gauss-Newton algorithm to avoid the computation of singular values for large matrices, so our algorithm is faster than those algorithms that require SVD. Except for problem~\eqref{pro:general}, this algorithm can be generalized to solve many other variants. 

\subsection{Nonconvex penalties on the singular values}
In the problem~\eqref{pro:general}, we choose the convex nuclear norm for the low-rank component in the objective function, which is the $\ell_1$ norm on the singular values. The $\ell_1$ norm pushes all singular values toward zero for the same amount, bringing bias in the solution. To promote the low-rankness of the low-rank component (or sparsity of its singular values), we can choose nonconvex regularization terms for the singular values. The idea for nonconvex regularization is to reduce the bias by pushing less on larger singular values. Some examples of nonconvex regularization are $\ell_p$ ($0\leq p<1$)~\cite{chartrand2007exact}, smoothly clipped absolute deviation (SCAD)~\cite{fan2001variable}, minimax concave penalty (MCP)~\cite{zhang2010nearly}, nonconvex weighted $\ell_1$~\cite{huang2015nonconvex}, etc. When these regularization terms are applied, the only difference is in the third step for finding $\vX$ in Lemma~\ref{lem2.3}. Currently, we have to apply the soft thresholding on the singular values. When nonconvex regularization is used, we apply the corresponding thresholding on the singular values. In this case, all the convergence results stay valid. 

\subsection{Other regularization on the sparse component}
We can also replace the $\ell_1$ norm of the sparse component with other regularization terms. Similarly to the penalty on the singular values, the $\ell_1$ norm on the sparse component brings bias, and we can use nonconvex regularization terms. Paper~\cite{wen2019nonconvex} uses both nonconvex regularization terms for the low-rank and sparse components. When different regularization terms are used on the sparse component, the new function $f_\lambda$ (see~\eqref{eq:flambda} for the definition) may not be differentiable any more. In this case, the convergence results do not hold. 

\subsection{Constrained problems}
When there is no noise in the measurements, the problem becomes constrained, and the previous algorithm can not be applied directly. Reference~\cite{Shen018} uses the penalty method and gradually increases the weight for the penalization to approximate the constrained problem. Here, we introduce a new method based on ADMM. We consider the following constrained problem
\begin{equation}\label{pro:model_C}
\Min_{\vL,\vS }~{\mu}\|\vL\|_*+  \|\vS\|_1, ~\St~ \mbox{rank}(\vL)\leq p,~\vD=\vL+\vS.
\end{equation}
When we apply ADMM, the steps are 
\begin{subequations}
	\begin{align}
	\vL^{k+1} = &~\argmin_{\vL:\mbox{rank}(\vL)\leq p} {\mu}\|\vL\|_* + {\alpha\over2}\|\vD-\vL-\vS^k+\frac{\vZ^k}{\alpha}\|_F^2;\\
	\vS^{k+1} = &~\argmin_\vS  \|\vS\|_1 + {\alpha\over2}\|\vD-\vL^{k+1}-\vS+\frac{\vZ^k}{\alpha}\|_F^2;\\
	\vZ^{k+1} = &~\vZ^k - \alpha (\vL^{k+1}+\vS^{k+1}-D).
	\end{align}
\end{subequations}
The first step is exactly the proximal operator that can be solved from Lemma~\ref{lem2.3}. The other two steps are easy to compute. This algorithm has only one parameter $\alpha$, while penalty methods, such as that in~\cite{Shen018}, require additional parameters to increase the weight for the penalization.

\section*{Acknowledgement} The authors thank Dr. Yuan Shen for sharing the code of the algorithm proposed in~\cite{Shen018}. 
The authors would like to thank two anonymous reviewers for their helpful comments and suggestions.


\medskip
Received xxxx 20xx; revised xxxx 20xx.
\medskip

\end{document}